\documentclass[reqno]{amsart}
\usepackage{amsmath}
\usepackage{paralist}
\usepackage{graphics} 
\usepackage{epsfig} 
\usepackage{graphicx}  \usepackage{epstopdf}
\usepackage{amssymb}
\usepackage{amsthm}
\usepackage{epstopdf}
\usepackage{verbatim}
\usepackage{mathrsfs}
\usepackage{mathtools}
\usepackage{pstricks}
\usepackage{relsize}
\usepackage{tikz}
\usetikzlibrary{matrix}
\usepackage{subcaption}
\usepackage{pgfplots}
\usepackage{hyperref}
\usepackage{cleveref}

\newtheorem{theorem}{Theorem}[section]
\newtheorem{corollary}{Corollary}

\newtheorem{lemma}[theorem]{Lemma}
\newtheorem{proposition}{Proposition}

\theoremstyle{definition}
\newtheorem{definition}[theorem]{Definition}

\crefname{equation}{}{} 
\Crefname{equation}{}{} 

\title[Nonuniqueness of minimizers] 
{Nonuniqueness of minimizers for semilinear optimal control problems}

\author[Dario Pighin]{}


\email{dario.pighin@uam.es}

\thanks{This project has received funding from the European Research Council (ERC) under the European Union’s Horizon 2020 research and innovation programme (grant agreement No 694126-DYCON).\\
	We acknowledge professor Enrique Zuazua for his helpful remarks on the manuscript. We thank professor Martin Gugat for his interesting questions.}

\begin{document}
	\maketitle
	
	\centerline{\scshape Dario Pighin}
	\medskip
	{\footnotesize
		\centerline{Departamento de Matem\'aticas, Universidad Aut\'onoma de Madrid}
		\centerline{28049 Madrid, Spain}
	} 
	\medskip
	{\footnotesize
		\centerline{Chair of Computational Mathematics, Fundaci\'on Deusto}
		\centerline{University of Deusto, 48007, Bilbao, Basque Country, Spain}
	} 

	\bigskip
	



\begin{abstract}
	A counterexample to uniqueness of global minimizers of semilinear optimal control problems is given. The lack of uniqueness occurs for a special choice of the state-target in the cost functional. Our arguments show also that, for some state-targets, there exist local minimizers, which are not global. When this occurs, gradient-type algorithms may be trapped by the local minimizers, thus missing the global ones. Furthermore, the issue of convexity of quadratic functional in optimal control is analyzed in an abstract setting.\\
	As a Corollary of the nonuniqueness of the minimizers, a nonuniqueness result for a coupled elliptic system is deduced.\\
	Numerical simulations have been performed illustrating the theoretical results.\\
	We also discuss the possible impact of the multiplicity of minimizers on the turnpike property in long time horizons.
\end{abstract}

\bigskip

{\sl AMS 2010 subject classifications\rm.
	Primary 93C20; Secondary 35J47, 35J61.}
\bigskip

\section{Introduction}
\label{sec:intro}

We produce a counterexample to the uniqueness of the optimal control in semilinear control. Both the case of internal control and boundary control are considered. To fix ideas, we focus on the case of quadratic functional and semilinear governing state equation. However, our techniques are applicable to a wide range of optimal control problems governed by a nonlinear state equation.

\subsection{Lack of uniqueness of the minimizer}
\label{subsec:lackofuniqueness}

In the context of boundary control, we consider the control problem
\begin{equation}\label{functional_nouniqboundary}
\min_{u\in L^{\infty}(\partial B(0,R))}J(u)=\frac12\int_{\partial B(0,R)} |u|^2 d\sigma(x) +\frac{\beta}{2}\int_{B(0,R)} |y-z|^2 dx,
\end{equation}
where $u=u(x)$ is the control and $y=y(x)$ is the associated state, solution to the semilinear equation
\begin{equation}\label{semilinear_boundary_elliptic_1}
\begin{dcases}
-\Delta y+f(y)=0\hspace{2.8 cm} & \mbox{in} \hspace{0.10 cm}B(0,R)\\
y=u  & \mbox{on}\hspace{0.10 cm} \partial B(0,R).
\end{dcases}
\end{equation}
The space domain $B(0,R)$ is a ball of $\mathbb{R}^n$ centered at the origin of radius $R$, with $n=1,2,3$. The nonlinearity $f\in C^1\left(\mathbb{R}\right)\cap C^2\left(\mathbb{R}\setminus \left\{0\right\}\right)$ is strictly increasing and $f(0)=0$. The target $z\in L^{\infty}(B(0,R))$ and $\beta> 0$ is a penalization parameter. As $\beta$ increases, the distance between the optimal state and the target decreases.

In \cref{sec:appendix.Preliminaries for boundary control} we analyze the well-posedness of the state equation \cref{semilinear_boundary_elliptic_1} and the existence of a global minimizer $\overline{u}\in L^{\infty}(\partial B(0,R))$ for the functional $J$ defined above. As we shall see in the following result, for a special target, the global minimizer is not unique.

\begin{figure}[tbhp]
\centering
\begin{tikzpicture}[scale=1]
\draw [shading=radial,outer color=lightgray!30,inner color=white] (0,0) circle [radius=3];
\draw [line width=1.26, blue] (0,0) circle [radius=3];
\draw [<-] (-2.1, 2.1) -- (-2.5,2.5);
\draw [-] (-2.5,2.5) -- (-3.6,2.5);
\node [left] at (-3.6,2.5) {control domain};
\draw [<-] (0.6, -0.6) -- (2.5,-2.5);
\draw [-] (2.5,-2.5) -- (3.6,-2.5);
\node [right] at (3.6,-2.5) {observation domain};
\end{tikzpicture}
\caption{control and observation domains. The control domain is the blue boundary of the ball.}
\label{fig:3_nouniq}
\end{figure}
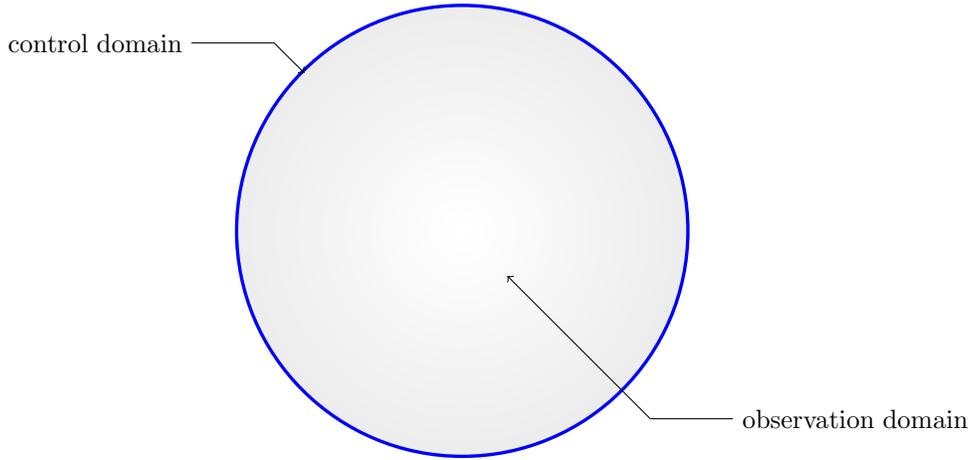

\begin{theorem}\label{th_nouniq_bound}
Consider the control problem \cref{semilinear_boundary_elliptic_1}-\cref{functional_nouniqboundary}. Assume, in addition
\begin{equation}\label{condition_boundary}
f^{\prime\prime}(y)\neq 0\hspace{1 cm}\forall \ y\neq 0.
\end{equation}
There exists a target $z\in L^{\infty}(B(0,R))$ such that the functional $J$ defined in \cref{functional_nouniqboundary} admits (at least) two global minimizers.
\end{theorem}

To give a first explanation of the above result, we introduce the control-to-state map
\begin{equation}\label{control_to_state_PDE_boundary_intro}
G:L^{\infty}(\partial B(0,R))\longrightarrow L^2(B(0,R))
\end{equation}
\begin{equation*}
u\longmapsto y_u,
\end{equation*}
with $y_u$ solution to \cref{semilinear_boundary_elliptic_1}, with control $u$. Then, for any control $u\in L^{\infty}(\partial B(0,R))$, the functional \cref{functional_nouniqboundary} reads as
\begin{equation}\label{functional_nouniqboundary_representation}
J(u) = \frac12\int_{\partial B(0,R)} \left|u\right|^2 d\sigma(x) +\frac{\beta}{2}\int_{B(0,R)} \left|G(u)-z\right|^2 dx.
\end{equation}
We have two addenda. The first one is convex, being a squared norm. The second one is a squared norm composed with $u\longmapsto G(u)-z$. Now, under the assumption \cref{condition_boundary}, the map $u\longmapsto G(u)$ is nonlinear. Then, the term $\int_{B(0,R)} \left|G(u)-z\right|^2 dx$, for a special target $z$, is not convex and generates the lack of uniqueness of the minimizers.

The proof of \Cref{th_nouniq_bound} can be found in \cref{subsec:nouniq.boundary}. The main steps for that proof are:
\begin{enumerate}
\item[Step 1] \textbf{Reduction to constant controls}: by choosing radial targets and using the rotational invariance of $B(0,R)$, we reduce to the case the control set is made of constant controls;
\item[Step 2] \textbf{Existence of two local minimizers}: we look for a target such that there exists two \textit{local} minimizers ($u_1<0$ and $u_2>0$)  for the functional $J$ (see \cref{fig:functionalnouniqmerged});
\item[Step 3] \textbf{Existence of two global minimizers}: by the former step and a bisection argument, we prove the existence of a target such that $J$ admits two \textit{global} minimizers.
\end{enumerate}

\begin{figure}
\centering
\begin{subfigure}{.5\textwidth}
	\centering
	\includegraphics[width=\textwidth]{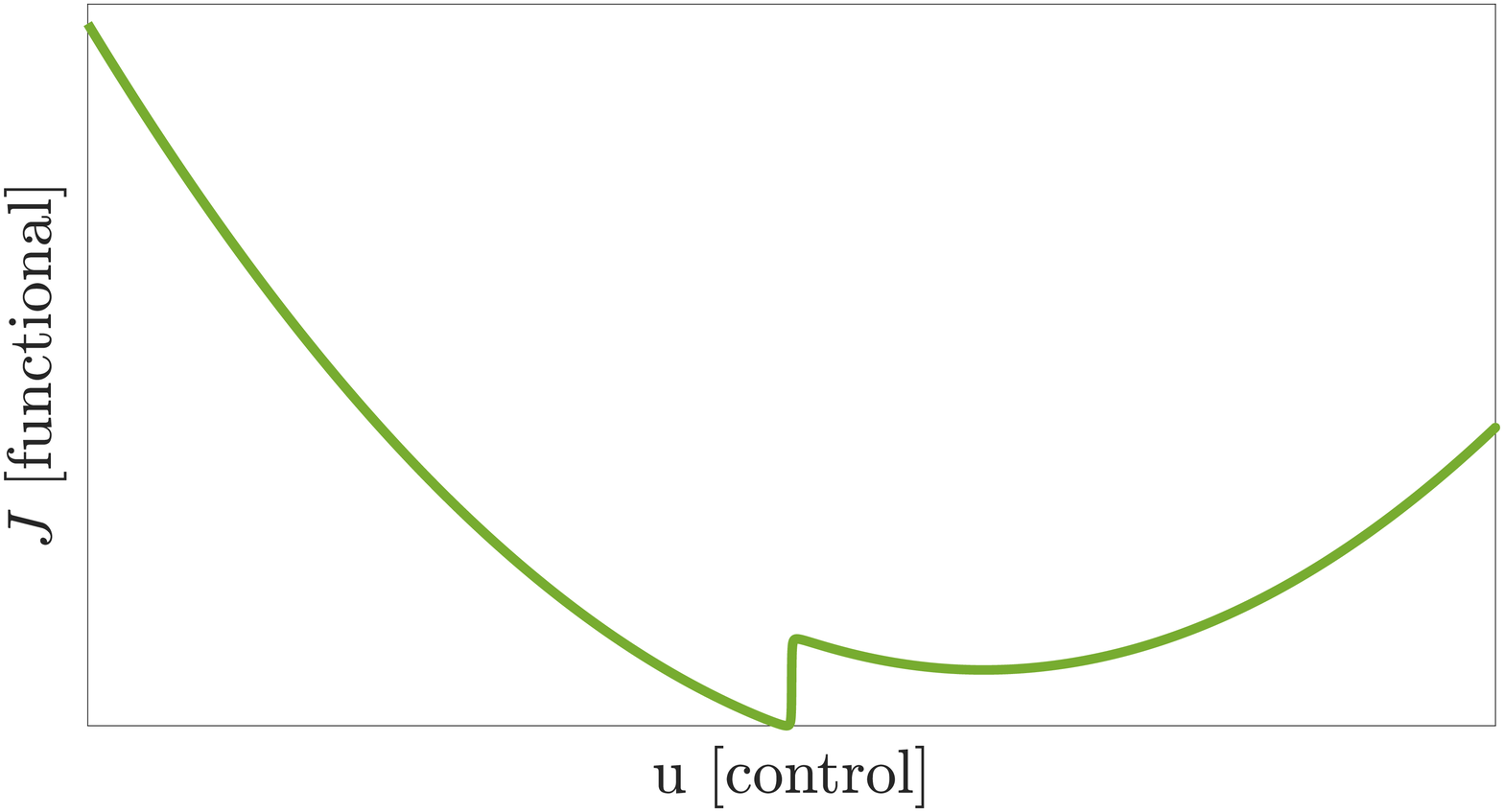}
	\caption{nonuniqueness of the \textit{local} minimizer}
	\label{fig:sub1}
\end{subfigure}%
\begin{subfigure}{.5\textwidth}
	\centering
	\includegraphics[width=\textwidth]{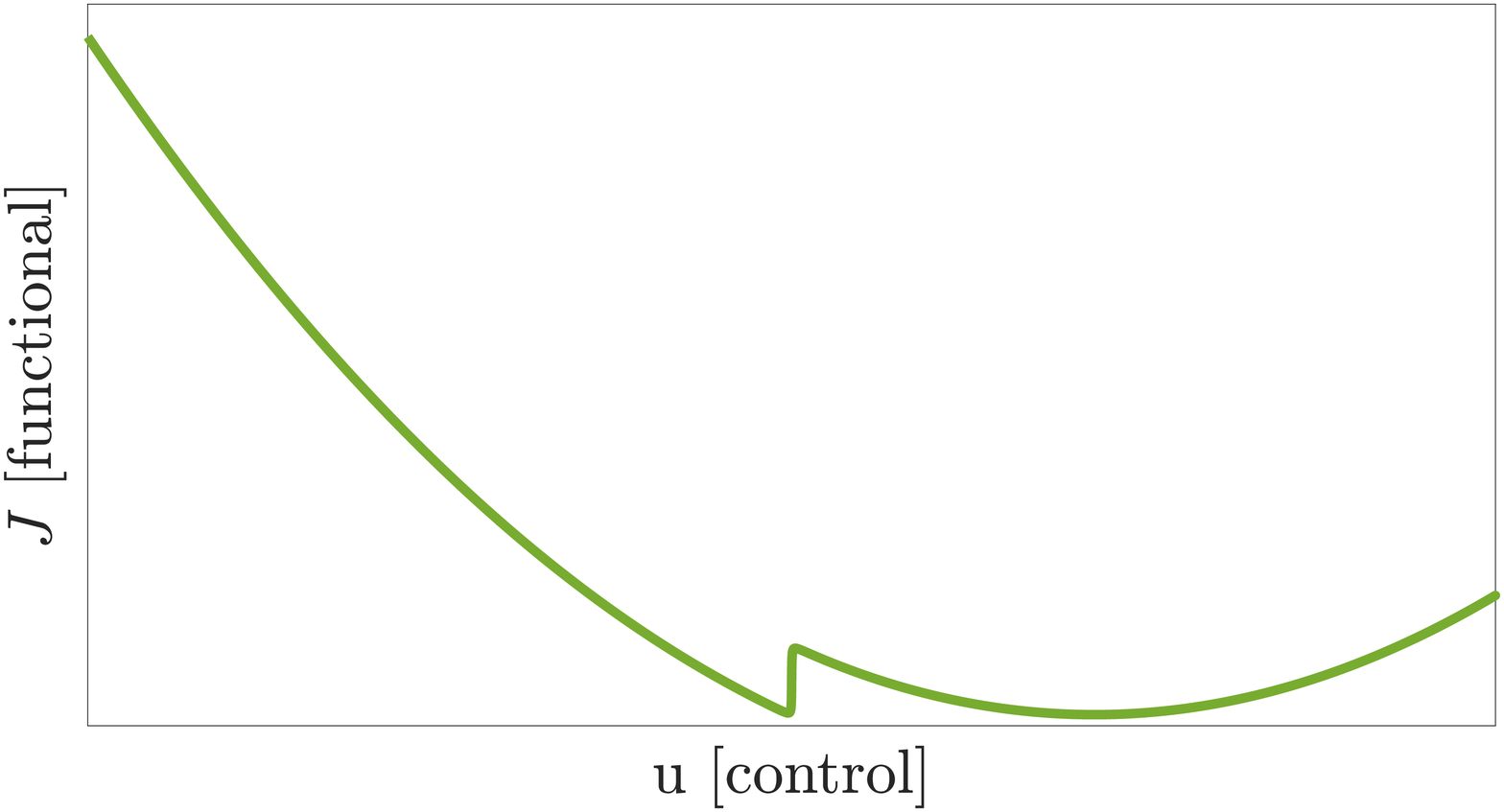}
	\caption{nonuniqueness of the \textit{global} minimizer}
	\label{fig:sub2}
\end{subfigure}
\caption{functional versus control. This plot is obtained by drawing in MATLAB the graph of $J$ defined in \cref{functional_nouniqboundary}, with $R=1$ and nonlinearity $f(y)=y^3$. \Cref{fig:sub1} and \cref{fig:sub2} correspond respectively to targets yielding to nonuniqueness of the local and the global minimizers.}\label{fig:functionalnouniqmerged}
\end{figure}

The special target yielding nonuniqueness is a step function changing sign in the observation domain, as in \cref{fig:target_nouniq}.

\begin{figure}[tbhp]
\centering
\includegraphics[width=12cm]{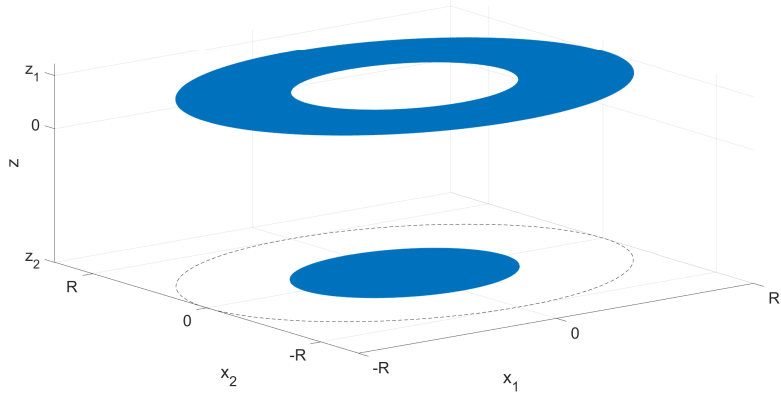}
\caption{target yielding nonuniqueness in boundary control. The constructed target $z$ (in blue) is a step function, taking values $z_1$ and $z_2$.}
\label{fig:target_nouniq}
\end{figure}

The above techniques can be applied, with some modifications, to the internal control problem
\begin{equation}\label{functionalinternal}
\min_{u\in L^2(B(0,r))}J(u)=\frac12\int_{B(0,r)} |u|^2 dx +\frac{\beta}{2}\int_{B(0,R)\setminus B(0,r)} |y-z|^2 dx,
\end{equation}
where
\begin{equation}\label{semilinear_internal_elliptic_1}
\begin{dcases}
-\Delta y+f(y)=u\chi_{B(0,r)}\hspace{2.8 cm} & \mbox{in} \hspace{0.10 cm}B(0,R)\\
y=0  & \mbox{on}\hspace{0.10 cm} \partial B(0,R).
\end{dcases}
\end{equation}
$B(0,R)$ denotes a ball of $\mathbb{R}^n$ centered at the origin of radius $R$, $n =1,2,3$. The nonlinearity $f\in C^1\left(\mathbb{R}\right)\cap C^2\left(\mathbb{R}\setminus \left\{0\right\}\right)$ is strictly increasing and $f(0)=0$. The control acts in $B(0,r)$, with $r\in (0,R)$. We observe in $B(0,R)\setminus B(0,r)$ (see \cref{fig:6}). The target $z\in L^2(B(0,R)\setminus B(0,r))$, while $\beta> 0$ is a penalization parameter.

The well-posedness of the state equation follows from \cite[Theorem 4.7, page 29]{boccardo2013elliptic}, while the existence of a global minimizer in $L^2(B(0,r))$ for \cref{semilinear_internal_elliptic_1}-\cref{functionalinternal} can be shown by the Direct Method of the Calculus of Variations (DMCV).

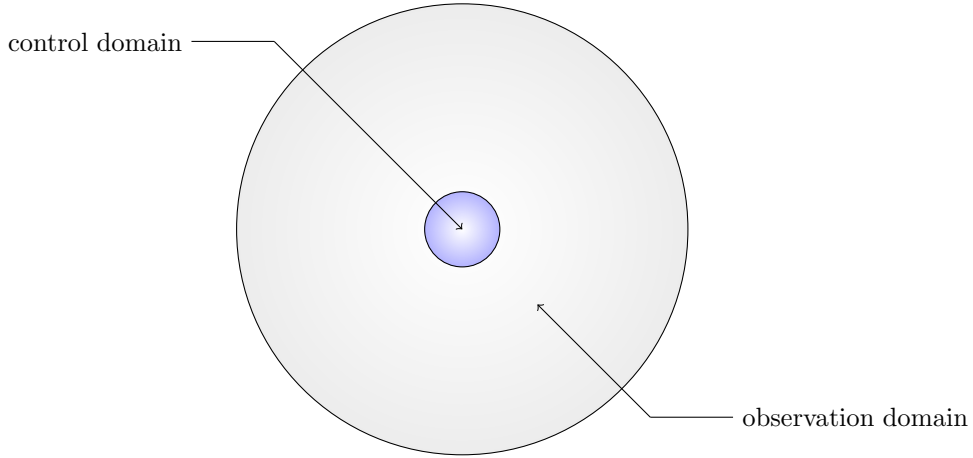
\begin{figure}[tbhp]
\centering
\begin{tikzpicture}[scale=1]
\draw [shading=radial,outer color=lightgray!30,inner color=white] (0,0) circle [radius=3];
\draw [shading=radial,outer color=blue!30,inner color=white] (0,0) circle [radius=0.5];
\draw [<-] (0, 0) -- (-2.5,2.5);
\draw [-] (-2.5,2.5) -- (-3.6,2.5);
\node [left] at (-3.6,2.5) {control domain};
\draw [<-] (1, -1) -- (2.5,-2.5);
\draw [-] (2.5,-2.5) -- (3.6,-2.5);
\node [right] at (3.6,-2.5) {observation domain};
\end{tikzpicture}
\caption{control and observation domains}
\label{fig:6}
\end{figure}

\begin{theorem}\label{th_nouniq_int}
Consider the control problem \cref{semilinear_internal_elliptic_1}-\cref{functionalinternal}. Assume, in addition,
\begin{equation}\label{conditioninternal}
f^{\prime\prime}(y)\neq 0\hspace{1 cm}\forall \ y\neq 0.
\end{equation}
There exists a target $z\in L^{\infty}(B(0,R)\setminus B(0,r))$ such that the functional $J$ defined in \cref{functionalinternal} admits (at least) two global minimizers.
\end{theorem}

The proof can be found in \cref{subsec:nouniq.internal}.

A by-product of our nonuniqueness results is the lack of uniqueness of solutions $\left(\overline{y},\overline{q}\right)$ to the optimality system
\begin{equation}\label{semilinear_internal_elliptic_2_nouniq}
\begin{dcases}
-\Delta \overline{y}+f(\overline{y})=-\overline{q}\chi_{B(0,r)}\hspace{2.8 cm} & \mbox{in} \hspace{0.10 cm}B(0,R)\\
\overline{y}=0  & \mbox{on}\hspace{0.10 cm} \partial B(0,R)\\
-\Delta \overline{q}+f^{\prime}(\overline{y})\overline{q}=\beta(\overline{y}-z)\chi_{B(0,R)\setminus B(0,r)}\hspace{2.8 cm} & \mbox{in} \hspace{0.10 cm}B(0,R)\\
\overline{q}=0  & \mbox{on}\hspace{0.10 cm} \partial B(0,R).
\end{dcases}
\end{equation}
In the case of internal control, we can deduce the following corollary.

\begin{corollary}\label{cor_nouniq_OS}
Under the assumptions of \Cref{th_nouniq_int}, there exists a target $z\in L^{\infty}(B(0,R)\setminus B(0,r))$, such that \cref{semilinear_internal_elliptic_2_nouniq} admits (at least) two distinguished solutions $\left(\overline{y}_1,\overline{q}_1\right)$ and $\left(\overline{y}_2,\overline{q}_2\right)$.
\end{corollary}
This follows from \Cref{th_nouniq_int}, together with the first order optimality conditions for the optimization problem \cref{semilinear_internal_elliptic_1}-\cref{functionalinternal} (see \cite{ESC}).

Similarly, in the context of boundary control, the nonuniqueness for \cref{functional_nouniqboundary} leads to nonuniquness of solution to the optimality system
\begin{equation}\label{semilinear_boundary_elliptic_2_nouniq}
\begin{dcases}
-\Delta \overline{y}+f(\overline{y})=0\hspace{2.8 cm} & \mbox{in} \hspace{0.10 cm}B(0,R)\\
\overline{y}=\frac{\partial}{\partial n}\overline{q}  & \mbox{on}\hspace{0.10 cm} \partial B(0,R)\\
-\Delta \overline{q}+f^{\prime}(\overline{y})\overline{q}=\beta(\overline{y}-z)\hspace{2.8 cm} & \mbox{in} \hspace{0.10 cm}B(0,R)\\
\overline{q}=0  & \mbox{on}\hspace{0.10 cm} \partial B(0,R).
\end{dcases}
\end{equation}

To the best of our knowledge, the issue of the uniqueness of the minimizer has not been addressed so far for large targets $z$. Indeed, the uniqueness of the optimal control has been proved under smallness conditions on the target \cite[subsection 3.2]{PZ2} or on the adjoint state \cite[Theorem 3.2]{ali2016global}. In particular, in \cite[Theorem 3.2]{ali2016global} the uniqueness holds provided that the adjoint state is strictly smaller than a constant, explicitly determined \cite[equation (3.6)]{ali2016global}.

The issue of uniqueness of the minimizer for elliptic problems is of primary importance when studying the turnpike property for the corresponding time-evolution control problem (see, 
\cite{trelat2015turnpike,PZ2,trelat2018steady,TGS}). Indeed, the existence of multiple global minimizers for the steady problem generates multiple potential attractors for the time-evolution problem.

The control problems we are treating are classical in the literature. General surveys on the topic are \cite{ESC} by Eduardo Casas and Mariano Mateos and \cite[Chapter 4]{troltzsch2010optimal} by Fredi Tr{\"o}ltzsch. The interested reader is refereed also to the following articles and books and the references therein: \cite{casas2014optimal,BNC,barbu2010nonlinear,neittaanmaki2007optimization,casas1993boundary,alibert1997boundary,haller2009holder,casas1986control,JOC,dontchev1995optimality,rosch2007regularity}.

\subsection{Lack of convexity}
\label{subsec:lackofconvexity}

Before proving our main result on nonuniqueness of global minimizers, we observe that, for some targets, quadratic functionals of the optimal control governed by nonlinear state equations are not convex.
\begin{theorem}\label{th_lackconv_intcontr}
Consider the optimal control problem introduced in \cref{semilinear_internal_elliptic_1}-\cref{functionalinternal}. Then, we have two possibilities:
\begin{enumerate}
	\item $f$ is linear. Then, $J$ is convex for any target $z\in L^2\left(B(0,R)\setminus B(0,r)\right)$.
	\item $f$ is not linear. Then, there exists a target $z\in L^2\left(B(0,R)\setminus B(0,r)\right)$ such that the corresponding $J$ is not convex.
\end{enumerate}
\end{theorem}

In the literature, it is well known that convexity cannot be proved by standard techniques, in case the state equation is nonlinear (see, for instance, \cite{ali2016global} and \cite[section 4]{troltzsch2010optimal}). However, to the best of our knowledge, there are not available counterexamples to convexity. In this work, the lack of convexity can be deduced as a consequence of the lack of uniqueness (\Cref{th_nouniq_bound}). Anyway, we prefer to prove \Cref{th_lackconv_intcontr} in \cref{sec:noconv} as a particular case of the following theorem, which holds in a general functional framework and basically asserts that a quadratic functional of the optimal control is convex for any target if and only if its control-to-state map is affine.

\begin{theorem}\label{thlackconv}
Let $U$ and $H$ be real Hilbert spaces. Let
\begin{equation*}
G:U\longrightarrow H
\end{equation*}
be a function. Set:
\begin{equation}\label{general_functional}
J:U\longrightarrow H,\hspace{0.3 cm}J(u)\coloneqq \frac12 \|u\|_U^2+\frac12 \|G(u)-z\|_H^2,
\end{equation}
where $z\in H$.\\
Then, the following are equivalent:
\begin{enumerate}
	\item for any target $z\in H$, $J$ is convex;
	\item $G$ is affine.
\end{enumerate}
\end{theorem}

In the application of \Cref{thlackconv} to optimal control, $H$ is the observation space, $U$ is the control space and $G$ is the control-to-state map. The vector $z\in H$ is the given target for the state. Note that \Cref{thlackconv} applies both to steady and time-evolution control problems. Furthermore, the map $G$ is not required to be smooth.

We sketch the proof of $1. \implies 2.$. Namely we show the lack of convexity, in case the control-to-state map $G$ is not affine. For the time being, we assume that $G$ is of class $C^2$. In the complete proof in \cref{sec:noconv}, the smoothness of $G$ is not required.

We start developing the functional \cref{general_functional}, for any control $u\in U$
\begin{align*} 
J(u) &= \frac12\left\|u\right\|_U^2 +\frac{1}{2} \left\|G(u)-z\right\|^2 \\
&= \frac12 \left\|u\right\|_{U}^2 +\frac{1}{2}\left\|G(u)\right\|_{H}^2+\frac{1}{2}\left\|z\right\|_{H}^2-\langle G(u),z\rangle \nonumber\\
&= P(u)+\frac{1}{2}\left\|z\right\|_{H}^2-\langle G(u),z\rangle, \nonumber\\
\end{align*}
where $\langle\cdot,\cdot\rangle$ denotes the scalar product of $H$ and
\begin{equation*}
P(u)\coloneqq \frac12 \left\|u\right\|_{U}^2 +\frac{1}{2}\left\|G(u)\right\|_{H}^2.
\end{equation*}

Now, since $G$ is not affine, there exists a control $u_1\in U$ and a direction $v_1\in U$, such that the second directional derivative of $G$ at $u_1$ along $v_1$ does not vanish
\begin{equation}
D^2G\left(u_1\right)\left(v_1,v_1\right)\neq 0.
\end{equation}
Take as target $z^k\coloneqq kD^2G\left(u_1\right)\left(v_1,v_1\right)$, with $k>0$ to be made precise later and compute the second differential of the functional $J$ at $u_1$ along direction $v_1$
\begin{align*}
\langle d^2J(u_1)v_1,v_1\rangle &= \frac{d^2}{dv_1^2}P(u_1)-\left\langle D^2G\left(u_1\right)\left(v_1,v_1\right),z^k\right\rangle \\
&= \frac{d^2}{dv_1^2}P(u_1)-k\left\|D^2G\left(u_1\right)\left(v_1,v_1\right)\right\|_H^2 <0, \nonumber\\
\end{align*}
choosing $k$ sufficiently large. This shows the lack of convexity in the smooth case. The general nonsmooth case is handled in \cref{sec:noconv}.

\Cref{thlackconv} can be applied to internal and boundary control, both in the elliptic and parabolic context.

The lack of convexity and uniqueness of the minimizer is a serious warning for numerics. Indeed, if the problem is not convex the convergence of gradient methods is not guaranteed a priori. Furthermore, by employing our techniques, one can find several counterexamples where there exist local minimizers, which are not global. Then, gradient methods may converge to the local minimizer, thus missing the global ones.

The rest of the manuscript is organized as follows. In \cref{sec:noconv}, we prove \Cref{thlackconv} and we deduce \Cref{th_lackconv_intcontr}. In \cref{sec:nouniq}, we provide the counterexample to uniqueness of the global minimizer, in the context of boundary control (\cref{subsec:nouniq.boundary}) and internal control (\cref{subsec:nouniq.internal}). In \cref{sec:simulations}, we perform numerical simulations which explain and confirm our theoretical results. In the appendix, we prove some Lemmas needed for our construction.

%

\section{Lack of convexity: proof of \Cref{thlackconv} and \Cref{th_lackconv_intcontr}}
\label{sec:noconv}

In the proof of \Cref{thlackconv}, we need the following lemma.

\begin{lemma}\label{lemma_2}
Let $V_1$ and $V_2$ be two real vector spaces. Take a function
\begin{equation*}
G:V_1\longrightarrow V_2.
\end{equation*}
Then, $G$ is affine if and only if, for any $\lambda\in [0,1]$ and $(v,w)\in V_1^2$
\begin{equation}\label{lemma_2_eq1}
G((1-\lambda)v+\lambda w)=(1-\lambda)G(v)+\lambda G(w).
\end{equation}
\end{lemma}
The proof can be deduced by linear algebra theory. We prove now \Cref{thlackconv}.

\begin{proof}[Proof of \Cref{thlackconv}]
$2.\implies 1.$ If $G$ is affine, by direct computations and convexity of the square of Hilbert norms, $J$ is convex for any $z\in H$.

$1. \implies 2.$ Assume now $G$ is not affine. We construct a target $z\in H$ such that $J$ is not convex.

In what follows, we denote by $\langle \cdot,\cdot \rangle $ the scalar product of $H$.\\
\textit{Step 1} \ \textbf{Proof of the existence of $\tilde{\lambda}\in [0,1]$, $(\tilde{u}_1,\tilde{u}_2)\in U^2$ and $z^0\in H$ such that:
	\begin{equation*}
	\left\langle z^0,G\left(\left(1-\tilde{\lambda}\right)\tilde{u}_1+\tilde{\lambda} \tilde{u}_2\right)\right\rangle < \left(1-\tilde{\lambda}\right)\left\langle z^0,G\left(\tilde{u}_1\right)\right\rangle +\tilde{\lambda}\left\langle z^0,G\left(\tilde{u}_2\right)\right\rangle  
	\end{equation*}}\\
First of all, we note that, up to change the sign of $z^0$, we can reduce to prove the existence of $\tilde{\lambda}\in [0,1]$, $\left(\tilde{u}_1,\tilde{u}_2\right)\in U^2$ and $z^0\in H$ such that:
\begin{equation}\label{thlackconv_eq6}
\left\langle z^0,G\left(\left(1-\tilde{\lambda}\right)\tilde{u}_1+\tilde{\lambda} \tilde{u}_2\right)\right\rangle \neq \left(1-\tilde{\lambda}\right)\langle z^0,G\left(\tilde{u}_1\right)\rangle +\tilde{\lambda}\left\langle z^0,G(\tilde{u}_2)\right\rangle . 
\end{equation}
Reasoning by contradiction, if \cref{thlackconv_eq6} were not true, for any $z\in H$, for every $(u_1,u_2)\in U^2$ and for each $\lambda\in [0,1]$,
\begin{equation*}
\left\langle z,G\left(\left(1-\lambda\right)u_1+\lambda u_2\right)\right\rangle = \left(1-\lambda\right)\left\langle z,G(u_1)\right\rangle +\lambda\left\langle z,G\left(u_2\right)\right\rangle . 
\end{equation*}
By the arbitrariness of $z$, this leads to:
\begin{equation*}
G\left(\left(1-\lambda\right)u_1+\lambda u_2\right)=\left(1-\lambda\right)G\left(u_1\right)+\lambda G\left(u_2\right), 
\end{equation*}
for any $\lambda\in [0,1]$ and $\left(u_1,u_2\right)\in U^2$. Then, by \Cref{lemma_2}, $G$ is affine, which contradicts our hypothesis. This finishes this step.\\
\textit{Step 2} \ \textbf{Conclusion}\\
We remind that in the first step, we have proved the existence of $\tilde{\lambda}\in [0,1]$, $\left(\tilde{u}_1,\tilde{u}_2\right)\in U^2$ and $z^0\in H$ such that:
\begin{equation*}
\left\langle z^0,G\left(\left(1-\tilde{\lambda}\right)\tilde{u}_1+\tilde{\lambda} \tilde{u}_2\right)\right\rangle < \left(1-\tilde{\lambda}\right)\left\langle z^0,G\left(\tilde{u}_1\right)\right\rangle +\tilde{\lambda}\left\langle z^0,G\left(\tilde{u}_2\right)\right\rangle .
\end{equation*}
Now, arbitrarily fix $k\in\mathbb{N}^*$. Set as target:
\begin{equation*}
z^k\coloneqq kz^0.
\end{equation*}
We develop $J$ with target $z^k$, getting for any $u\in U$:
\begin{align*} 
J\left(u\right) &= \frac12\left\|u\right\|_U^2+\frac12\left\|G(u)-z^k\right\|_H^2 \\
&=\frac12\left\|u\right\|_U^2+\frac12\left\|G(u)\right\|_H^2+\frac12 \left\|z^k\right\|_H^2-\left\langle z^k,G(u)\right\rangle \nonumber\\
&=P(u)+\frac12 \left\|z^k\right\|_H^2-\left\langle z^k,G(u)\right\rangle ,
\end{align*}
where
\begin{equation*}
P:U\longrightarrow \mathbb{R},\hspace{0.3 cm}u\longmapsto \frac12\left\|u\right\|_U^2+\frac12\left\|G(u)\right\|_H^2.
\end{equation*}
At this point, we introduce:
\begin{equation*}
c_1\coloneqq \left(1-\tilde{\lambda}\right)P\left(\tilde{u}_1\right)+\tilde{\lambda} P\left(\tilde{u}_2\right)-P\left(\left(1-\tilde{\lambda}\right)\tilde{u}_1+\tilde{\lambda} \tilde{u}_2\right)
\end{equation*}
and
\begin{equation*}
c_2	\coloneqq \left(1-\tilde{\lambda}\right)\langle z^0,G\left(\tilde{u}_1\right)\rangle +\tilde{\lambda}\left\langle z^0,G\left(\tilde{u}_2\right)\right\rangle -\left\langle z^0,G\left( \left(1-\tilde{\lambda}\right)\tilde{u}_1+\tilde{\lambda} \tilde{u}_2\right)\right\rangle . 
\end{equation*}
Then, taking as target $z^k$,
\begin{equation*}
\left(1-\tilde{\lambda}\right)J\left(\tilde{u}_1\right)+\tilde{\lambda}J\left(\tilde{u}_2\right)-J\left(\left(1-\tilde{\lambda}\right)\tilde{u}_1+\tilde{\lambda} \tilde{u}_2\right)=c_1-kc_2.
\end{equation*}
By the first step, $c_2>0$. Then, for $k$ large enough, we have:
\begin{equation*}
\left(1-\tilde{\lambda}\right)J\left(\tilde{u}_1\right)+\tilde{\lambda}J\left(\tilde{u}_2\right)-J\left(\left(1-\tilde{\lambda}\right)\tilde{u}_1+\tilde{\lambda} \tilde{u}_2\right)=c_1-kc_2<0,
\end{equation*}
which yields
\begin{equation*}
\left(1-\tilde{\lambda}\right)J\left(\tilde{u}_1\right)+\tilde{\lambda}J\left(\tilde{u}_2\right)<J\left(\left(1-\tilde{\lambda}\right)\tilde{u}_1+\tilde{\lambda} \tilde{u}_2\right),
\end{equation*}
i.e. the desired lack of convexity of $J$. This concludes the proof.
\end{proof}

\Cref{thlackconv} applies in semilinear control, both in the elliptic case and in the parabolic one. We show how to apply \Cref{thlackconv} for the control problem \cref{semilinear_internal_elliptic_1}-\cref{functionalinternal}, thus proving \Cref{th_lackconv_intcontr}.

\begin{proof}[Proof of \Cref{th_lackconv_intcontr}]
Take
\begin{itemize}
	\item control space $U=L^2(B(0,r))$;
	\item $H=L^2\left(B(0,R)\setminus B(0,r)\right)$ with scalar product\\
	$\langle v_1,v_2\rangle\coloneqq \beta\int_{B(0,R)\setminus B(0,r)}v_1v_2dx$; 
	\item the map
	\begin{equation*}
	G:L^2(B(0,r))\longrightarrow L^2\left(B(0,R)\setminus B(0,r)\right)
	\end{equation*}
	\begin{equation*}
	u\longrightarrow y_u\hspace{-0.1 cm}\restriction_{B(0,R)\setminus B(0,r)},
	\end{equation*}
	where $y_u$ fulfills \cref{semilinear_internal_elliptic_1} with control $u$.
\end{itemize}
Then, by \Cref{thlackconv}, we have two possibilities:
\begin{enumerate}
	\item $G$ is linear. Then, $J$ is convex for any target $z\in L^2\left(B(0,R)\setminus B(0,r)\right)$.
	\item $G$ is not linear. Then, there exists a target $z\in L^2\left(B(0,R)\setminus B(0,r)\right)$ such that the corresponding $J$ is not convex.
\end{enumerate}
It remains to prove that $G$ is linear if and only if $f$ is linear. Now, if $f$ is linear, the linearity of $G$ follows from linear PDE theory \cite[Part I]{EPG}. Suppose now $G$ is linear. Let us prove that $f$ is linear, namely for any $\alpha$, $\beta$, $\theta_1$ and $\theta_2\in \mathbb{R}$
\begin{equation}\label{f_linear}
f\left(\alpha \theta_1+\beta\theta_2\right)=\alpha f\left( \theta_1\right)+\beta f\left(\theta_2\right).
\end{equation}
To this extent, let us introduce a cut-off function $\zeta \in C^{\infty}(\mathbb{R}^n)$ such that:
\begin{itemize}
	\item $\zeta(0)=1$;
	\item $\mbox{supp}(\zeta)\subset\subset B(0,r)$.
\end{itemize}
For $i=1,2$, set $y_{\theta_i}\coloneqq \theta_i \zeta$ and $u_{\theta_i}\coloneqq \left[-\Delta y_{\theta_i}+f\left(y_{\theta_i}\right)\right]\hspace{-0.1 cm}\restriction_{B(0,r)}$. Then, by the linearity of $G$
\begin{eqnarray}\label{}
f\left(\alpha y_{\theta_1}+\beta y_{\theta_2}\right)&=&f\left(\alpha G\left(u_{\theta_1}\right)+\beta G\left(u_{\theta_2}\right)\right)\nonumber\\
&=&f\left(G\left(\alpha u_{\theta_1}+\beta u_{\theta_2}\right)\right)\nonumber\\
&=&\Delta G\left(\alpha u_{\theta_1}+\beta u_{\theta_2}\right)+\left(\alpha u_{\theta_1}+\beta u_{\theta_2}\right)\chi_{B(0,r)}\nonumber\\
&=&\alpha\Delta G\left( u_{\theta_1}\right)+\beta \Delta G\left(u_{\theta_2}\right)+\alpha u_{\theta_1}\chi_{B(0,r)}+\beta u_{\theta_2}\chi_{B(0,r)}\nonumber\\
&=&\alpha f\left( y_{\theta_1}\right)+\beta f\left(y_{\theta_2}\right),
\end{eqnarray}
whence
\begin{eqnarray}\label{}
f\left(\alpha \theta_1+\beta\theta_2\right)&=&f\left(\alpha y_{\theta_1}(0)+\beta y_{\theta_2}(0)\right)\nonumber\\
&=&\alpha f\left( y_{\theta_1}(0)\right)+\beta f\left(y_{\theta_2}(0)\right)\nonumber\\
&=&\alpha f\left( \theta_1\right)+\beta f\left(\theta_2\right),
\end{eqnarray}
as required.
\end{proof}

\section{Lack of uniqueness}
\label{sec:nouniq}

In this section, we prove our nouniqueness results. We start with boundary control (\Cref{th_nouniq_bound}), to later deal with internal control (\Cref{th_nouniq_int}).

\subsection{Boundary control}
\label{subsec:nouniq.boundary}

Hereafter, we will work with radial targets, defined below.
\begin{definition}\label{Def_radial_target_boundary}
A function $z:B(0,R)\longrightarrow \mathbb{R}$ is said to be radial if there exists $\phi:[0,R]\longrightarrow \mathbb{R}$, such that, for any $x\in B(0,R)$, we have $z(x)=\phi(\|x\|)$.
\end{definition}

\begin{figure}[tbhp]
\centering
\includegraphics[width=13 cm]{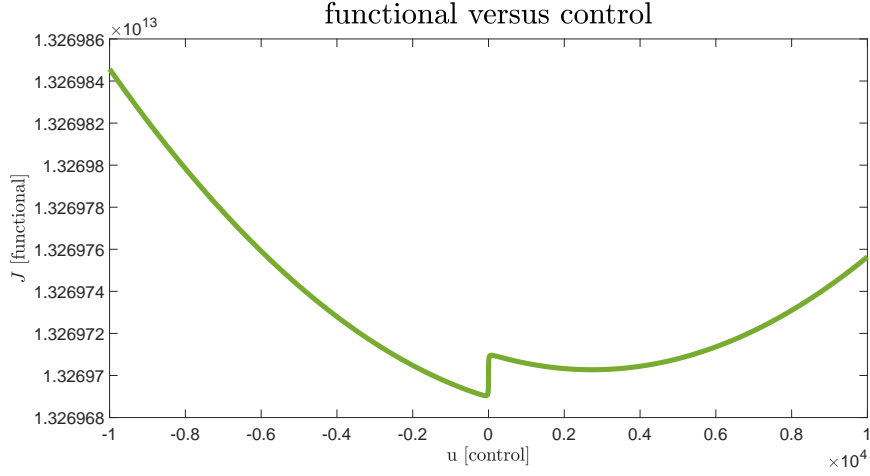}
\caption{functional versus control (nonuniqueness of the \textit{local} minimizer). This plot is obtained by drawing in MATLAB the graph of $J$ defined in \cref{functional_nouniqboundary}, with $R=1$ and nonlinearity $f(y)=y^3$. The target $z=260000\chi_{\left(0,\frac14\right)\cup \left(\frac34,1\right)}-10300000\chi_{\left(\frac14,\frac34\right)}$.}
\label{fig:functionalnouniqlocal}
\end{figure}

\begin{figure}[tbhp]
\begin{center}
	\centering
	\includegraphics[width=13 cm]{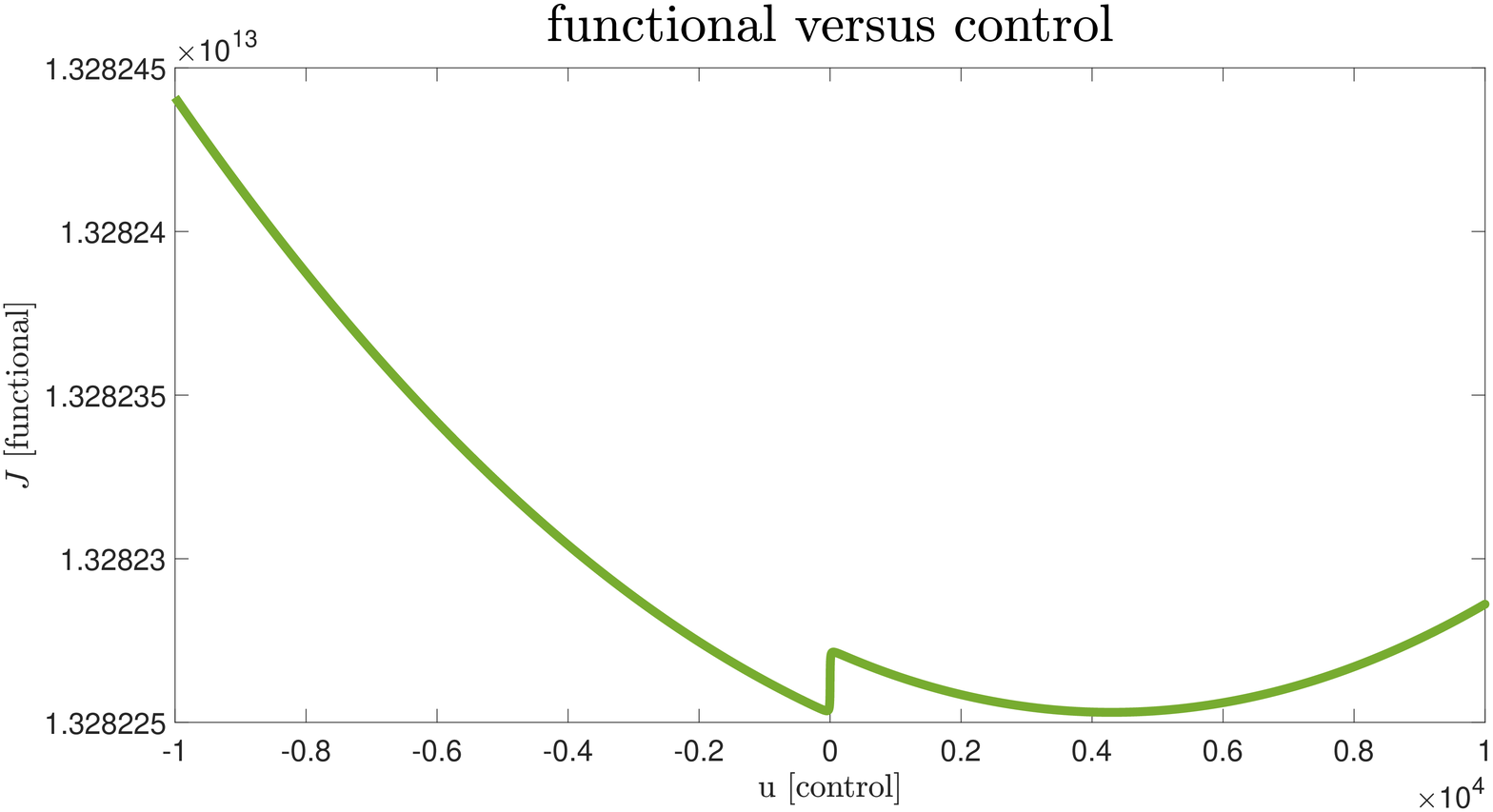}
	\caption{functional versus control (nonuniqueness of the \textit{global} minimizer). This plot is obtained by drawing in MATLAB the graph of $J$ defined in \cref{functional_nouniqboundary}, with $R=1$ and nonlinearity $f(y)=y^3$. The target $z=410000\chi_{\left(0,\frac14\right)\cup \left(\frac34,1\right)}-10300000\chi_{\left(\frac14,\frac34\right)}$.}
	\label{fig:functionalnouniq1}
\end{center}
\end{figure}

We introduce the control-to-state map
\begin{equation}\label{control_to_state_PDE_boundary}
G:L^{\infty}(\partial B(0,R))\longrightarrow L^2(B(0,R))
\end{equation}
\begin{equation*}
u\longmapsto y_u,
\end{equation*}
where $y_u$ is the solution to \cref{semilinear_boundary_elliptic_1} with control $u$. Then, set:
\begin{equation}\label{def_functional_control_target_PDE_boundary}
I:L^{\infty}(\partial B(0,R))\times L^{2}(B(0,R))\longrightarrow \mathbb{R}
\end{equation}
\begin{equation*}
I(u,z)\coloneqq \frac12\int_{\partial B(0,R)}|u|^2d\sigma(x)+\frac{\beta}{2}\int_{B(0,R)} |G(u)|^2dx-\beta\int_{B(0,R)} G(u)zdx,
\end{equation*}
where $G$ is the control-to-state map introduced in \cref{control_to_state_PDE_boundary}. One recognizes that, for any target $z\in L^{\infty}(B(0,R))$, $I(\cdot,z)+\frac{{\beta}}{2} \|z\|_{L^2(B(0,R))}^2$ coincides with the functional $J$ defined in \cref{functional_nouniqboundary} with target $z$. Then, for any target $z\in L^{\infty}(B(0,R))$ minimizing $I(\cdot, z)$ is equivalent to minimizing $J$ with target $z$. Such translation is convenient, because $I(0,z)=0$ for any target $z\in L^{\infty}(B(0,R))$.

We establish some important properties of the solutions of the state equation \cref{semilinear_boundary_elliptic_1}:
\begin{itemize}
\item The unique constant solution of the equation $-\Delta y+f(y) = 0$ in any domain $\Omega\subset B(0,R)$ is $y \equiv 0$ (\Cref{lemma_rangesolutionelliptic}). In particular, $G(u) = 0$ if and only if $u = 0$ holds.
\item By comparison principle,
if $u \geq 0$ on $\partial B(0,R)$ and $u \not\equiv 0$, then $G(u)(x) > 0$ in $B(0,R)$.
\item By comparison principle,
if $u \leq 0$ on $\partial B(0,R)$ and $u \not\equiv 0$, then $G(u)(x) < 0$ in $B(0,R)$.
\end{itemize}

We introduce:
\begin{equation}\label{def_h_1_PDE_boundary}
h_1:L^{\infty}(B(0,R))\longrightarrow \mathbb{R}, \hspace{0.3 cm}h_1(z)\coloneqq \inf\left\{I(u,z) \ | \ u\equiv k, \ k\in (-\infty,0]\right\}
\end{equation}
and
\begin{equation}\label{def_h_2_PDE_boundary}
h_2:L^{\infty}(B(0,R))\longrightarrow \mathbb{R},\hspace{0.3 cm}h_2(z)\coloneqq \inf\left\{I(u,z) \ | \ u\equiv k, \ k\in [0,+\infty)\right\}.
\end{equation}

We formulate the first lemma.

\begin{lemma}\label{lemma3_boundary}
Let $C=\left(-\infty,0\right]$ or $C=\left[0,+\infty\right)$. Then,
\begin{enumerate}
	\item for any $z\in L^{\infty}(B(0,R))$, there exists $u_{z}\in C$ such that:
	\begin{equation*}
	I(u_{z},z)=\inf_{C}[I(\cdot,z)].
	\end{equation*}
	Furthermore, for any minimizer $u_z$
	\begin{equation*}
	|u_{z}|\leq \sqrt{\frac{\beta}{R^{n-1}n\alpha(n)}}\|z\|_{L^2},
	\end{equation*}
	where $n\alpha(n)$ is the surface area of $\partial B(0,1)\subset \mathbb{R}^n$, the unit sphere.
	\item the map
	\begin{equation*}
	h:L^{\infty}(B(0,R))\longrightarrow \mathbb{R},\hspace{0.3 cm} h(z)\coloneqq \inf_{C}\left[I(\cdot, z)\right]
	\end{equation*}
	is continuous.
\end{enumerate}
\end{lemma}
We prove \Cref{lemma3_boundary} in \cref{sec:appendix.Preliminaries for boundary control}. We now state the second lemma.

\begin{lemma}\label{lemma4_boundary}
Assume there exists $z^0\in L^{\infty}(B(0,R))$ such that
\begin{equation*}
h_1(z^0)<0\hspace{1 cm}\mbox{and}\hspace{1 cm}h_2(z^0)<0,
\end{equation*}
where $h_1$ and $h_2$ are defined in \cref{def_h_1_PDE_boundary} and \cref{def_h_2_PDE_boundary} resp.
Then, there exists a target $\tilde{z}\in L^{\infty}(B(0,R))$ such that
\begin{equation*}
h_1(\tilde{z})=h_2(\tilde{z})<0.
\end{equation*}
\end{lemma}

The proof of \Cref{lemma4_boundary} can be found in \cref{sec:appendix.Preliminaries for boundary control}. The following lemma is the key-point for the proof of the existence of two local minimizers for \cref{functional_nouniqboundary}. At this point we employ the nonlinearity of the state equation \cref{semilinear_boundary_elliptic_1}.

\begin{lemma}\label{lemma_rank_matrix_boundary}
Let $\Omega$ be a bounded open subset of $\mathbb{R}^n$, with $\partial \Omega \in C^{\infty}$. Let $u_{-}<0<u_{+,1}<u_{+,2}$ be three constant controls. For any $u\in L^{\infty}\left(\partial \Omega\right)$, let $G\left(u\right)$ be the solution to
\begin{equation}\label{semilinear_boundary_elliptic_1_gdomain_wp_2}
\begin{dcases}
-\Delta y+f(y)=0\hspace{2.8 cm} & \mbox{in} \hspace{0.10 cm}\Omega\\
y=u  & \mbox{on}\hspace{0.10 cm} \partial \Omega.
\end{dcases}
\end{equation}
Assume $f\in C^1\left(\mathbb{R}\right)\cap C^2\left(\mathbb{R}\setminus \left\{0\right\}\right)$ is strictly increasing and
\begin{equation*}\label{condition_boundary_3}
f^{\prime\prime}(y)\neq 0\hspace{1 cm}\forall \ y\neq 0.
\end{equation*}
Set
\begin{equation}
\overline{\lambda}\coloneqq \frac{\int_{\Omega}G\left(u_{+,2}\right)(x)dx}{\int_{\Omega}G\left(u_{+,1}\right)(x)dx},
\end{equation}
\begin{equation}\label{lemma_molt_defomega1}
\omega_1\coloneqq \left\{x\in \Omega \ | \ G\left(u_{+,2}\right)(x)<\overline{\lambda} \hspace{0.003 cm} G\left(u_{+,1}\right)(x)\right\}
\end{equation}
and
\begin{equation}\label{lemma_molt_defomega2}
\omega_2\coloneqq \left\{x\in \Omega \ | \ G\left(u_{+,2}\right)(x)>\overline{\lambda} \hspace{0.003 cm} G\left(u_{+,1}\right)(x)\right\}.
\end{equation}
There exist $i\in \left\{1,2\right\}$, such that
\begin{equation}\label{def_Gamma_boundary}
\Gamma\coloneqq \beta\begin{tikzpicture}[baseline=(current bounding box.center)]
\matrix (Gamma) [matrix of math nodes,nodes in empty cells,right delimiter={]},left delimiter={[} ]{
	\mathlarger{\int_{\omega_1}}G(u_{-})dx\hspace{0.3 cm}&\mathlarger{\int_{\omega_2}}G(u_{-})dx     \\
	\mathlarger{\int_{\omega_1}}G(u_{+,i})dx&\mathlarger{\int_{\omega_2}}G(u_{+,i}) dx  \\
} ;
\end{tikzpicture}.
\end{equation}
is invertible.
\end{lemma}


\begin{proof}[Proof of \Cref{lemma_rank_matrix_boundary}]
To simplify the notation, we set $y_1 \coloneqq G\left(u_{+,1}\right)$ and $y_2\coloneqq G\left(u_{+,2}\right)$.\\
\textit{Step 1} \  \textbf{For any $\lambda \in \mathbb{R}$ the set
	\begin{equation*}\label{set_y2lambday1}
	E_{\lambda}\coloneqq \left\{x\in \Omega \ | \ y_2(x)=\lambda y_1(x)\right\}
	\end{equation*}
	has Lebesgue measure zero.}\\
We start with the case $\lambda \leq 1$. By the strong maximum principle \cite[Theorem 8.19 page 198]{EPG}, for any $x\in \Omega$, $G\left(u_{+,2}\right)(x)> y_1(x)$. Hence, for any $\lambda \leq 1$, the set $E_{\lambda}$ defined in \cref{set_y2lambday1} is empty.

We conclude Step 1, with the case $\lambda >1$. Suppose, by contradiction, that $E_{\lambda}$ has strictly positive Lebesgue measure. For any $x\in \Omega$, we have
\begin{equation}\label{lemma_molt_eq2}
-\Delta y_1(x) + f\left( y_1(x)\right)=0
\end{equation}
and
\begin{equation}\label{lemma_molt_eq3}
-\Delta y_2(x) + f\left(y_2(x)\right)=0.
\end{equation}
By definition \cref{set_y2lambday1}, for any $x\in E_{\lambda}$, $y_2(x)=\lambda y_1(x)$, whence by \Cref{lemma_molt_eq3} and \Cref{lemma_funtion_grad} applied twice, we get a.e. in $E_{\lambda}$
\begin{equation}\label{lemma_molt_eq4}
-\lambda\Delta y_1(x) + f\left( \lambda y_1(x)\right)=0.
\end{equation}
Multiplying \cref{lemma_molt_eq2} by $\lambda$, we have
\begin{equation}\label{lemma_molt_eq6}
-\lambda\Delta y_1(x) + \lambda f\left( y_1(x)\right)=0.
\end{equation}
By subtracting \cref{lemma_molt_eq4} and \cref{lemma_molt_eq6}, we obtain
\begin{equation}\label{lemma_molt_eq8}
f\left( \lambda y_1(x)\right)=\lambda f\left(  y_1(x)\right),\hspace{2.8 cm} \hspace{0.10 cm}\mbox{a.e.} \ x\in E_{\lambda}.
\end{equation}
Now, we have supposed that $E_{\lambda}$ has a positive Lebesgue measure. Hence, by \Cref{lemma_posmeasndir}, there exists an accumulation point $\hat{x}\in \Omega$ and a corresponding sequence $\left\{x_m\right\}_{m\in\mathbb{N}}\subset E_{\lambda}$ such that
\begin{equation}
x_m\underset{m\to +\infty}{\longrightarrow}\hat{x}.
\end{equation}
Now, by \cref{lemma_molt_eq8}, we have
\begin{equation}\label{lemma_molt_eq9}
f\left( \lambda y_1\left(x_m\right)\right)=\lambda f\left(  y_1\left(x_m\right)\right),\hspace{2.8 cm} \hspace{0.10 cm}\forall m\in \mathbb{N}.
\end{equation}
Since $u_i\in C^0\left(\partial \Omega\right)$, it follows that $y_i\in H^1\left(\Omega\right)\cap C^0\left(\overline{\Omega}\right)$ by virtue of \Cref{prop_exuniqellipticsemilinear}. Then, taking the limit as $m\to +\infty$ in the above expression, we get
\begin{equation}\label{lemma_molt_eq10}
f\left( \lambda y_1\left(\hat{x}\right)\right)=\lambda f\left(  y_1\left(\hat{x}\right)\right).
\end{equation}
Hence, by \cref{lemma_molt_eq9} and \cref{lemma_molt_eq10}, we have
\begin{equation}
\frac{f(\lambda y_1(x_m))-f(\lambda y_1(\hat{x}))}{\lambda y_1(x_m)-\lambda y_1(\hat{x})}=\frac{\lambda f( y_1(x_m))-\lambda f( y_1(\hat{x}))}{\lambda y_1(x_m)-\lambda y_1(\hat{x})}.
\end{equation}
Taking the limit as $m\to +\infty$ in both sides, using the continuity of $y_1$ we get $f^{\prime}(\lambda y_1(\hat{x}))=f^{\prime}(y_1(\hat{x}))$. Now, by \cite[Theorem 8.19 page 198]{EPG}, $y_1(\hat{x})>0$. Hence by Rolle Theorem applied to $f^{\prime}$, there exists $\xi>0$ such that
\begin{equation}
f^{\prime\prime}(\xi)=0,
\end{equation}
so obtaining a contradiction with assumptions. This finishes Step 1.

Set now
\begin{equation*}
\Lambda\coloneqq \begin{tikzpicture}[baseline=(current bounding box.center)]
\matrix (Lambda) [matrix of math nodes,nodes in empty cells,right delimiter={]},left delimiter={[} ]{
	\mathlarger{\int_{\omega_1}}G\left(u_{+,1}\right)dx\hspace{0.3 cm}&\mathlarger{\int_{\omega_2}}G\left(u_{+,1}\right)dx     \\
	\mathlarger{\int_{\omega_1}}G\left(u_{+,2}\right)dx&\mathlarger{\int_{\omega_2}}G\left(u_{+,2}\right) dx  \\
} ;
\end{tikzpicture}
\end{equation*}
\textit{Step 2} \  \textbf{$\Omega\setminus \left[\omega_1\cup\omega_2\right]$ has Lebesgue measure zero and the matrix $\Lambda$ is invertible.}\\
By the above reasoning, the set $E_{\overline{\lambda}}=\Omega\setminus \left[\omega_1\cup\omega_2\right]$ has Lebesgue measure zero. Now, by the strong maximum principle, $y_1$ and $y_2$ are strictly positive in $\Omega$ and $\overline{\lambda}\neq 0$. Hence,
\begin{eqnarray*}
	\det\left(\Lambda\right)&=&\mathlarger{\int_{\omega_1}}y_1dx\mathlarger{\int_{\omega_2}}y_2 dx-\mathlarger{\int_{\omega_1}}y_2dx\mathlarger{\int_{\omega_2}}y_1dx\nonumber\\
	&>&\overline{\lambda}\mathlarger{\int_{\omega_1}}y_1dx\mathlarger{\int_{\omega_2}}y_1 dx-\overline{\lambda}\mathlarger{\int_{\omega_1}}y_1dx\mathlarger{\int_{\omega_2}}y_1dx=0.
\end{eqnarray*}
\textit{Step 3} \  \textbf{Conclusion}\\
Let us assume, by contradiction, that the matrix $\Gamma$ is not invertible. Then, for $i=1,2$, there exists $\lambda_i \in\mathbb{R}$ such that
\begin{equation}\label{contradd_PDE_boundary}
\begin{tikzpicture}[baseline=(current bounding box.center)]
\matrix (Gamma) [matrix of math nodes,nodes in empty cells,right delimiter={]},left delimiter={[} ]{
	\mathlarger{\int_{\omega_1}}G(u_{+,i})dx     \\
	\mathlarger{\int_{\omega_2}}G(u_{+,i}) dx  \\
} ;
\end{tikzpicture}=\lambda_i \begin{tikzpicture}[baseline=(current bounding box.center)]
\matrix (c1) [matrix of math nodes,nodes in empty cells,right delimiter={]},left delimiter={[} ]{
	\mathlarger{\int_{\omega_1}}G(u_{-})dx     \\
	\mathlarger{\int_{\omega_2}}G(u_{-})dx \\
} ;
\end{tikzpicture}.
\end{equation}
Since the controls are nonzero constants, by \cite[Theorem 8.19 page 198]{EPG}, all the above integrals do not vanish, whence $\lambda_i \neq 0$. Then, we have
\begin{equation}\label{contradd_PDE_boundary_2}
\begin{tikzpicture}[baseline=(current bounding box.center)]
\matrix (Gamma) [matrix of math nodes,nodes in empty cells,right delimiter={]},left delimiter={[} ]{
	\mathlarger{\int_{\omega_1}}G(u_{+,2})dx     \\
	\mathlarger{\int_{\omega_2}}G(u_{+,2}) dx  \\
} ;
\end{tikzpicture}=\lambda_2 \begin{tikzpicture}[baseline=(current bounding box.center)]
\matrix (c1) [matrix of math nodes,nodes in empty cells,right delimiter={]},left delimiter={[} ]{
	\mathlarger{\int_{\omega_1}}G(u_{-})dx     \\
	\mathlarger{\int_{\omega_2}}G(u_{-})dx \\
} ;
\end{tikzpicture}=\frac{\lambda_2}{\lambda_1} \begin{tikzpicture}[baseline=(current bounding box.center)]
\matrix (c1) [matrix of math nodes,nodes in empty cells,right delimiter={]},left delimiter={[} ]{
	\mathlarger{\int_{\omega_1}}G(u_{+,1})dx     \\
	\mathlarger{\int_{\omega_2}}G(u_{+,1})dx \\
} ;
\end{tikzpicture}.
\end{equation}
By \cref{contradd_PDE_boundary_2}, the matrix $\Lambda$ is not invertible, so obtaining a contradiction with Step 3.
\end{proof}

\begin{proof}[Proof of \Cref{th_nouniq_bound}.]
\textit{Step 1} \ \textbf{Reduction to constant controls.}\\
Suppose for some radial target $z$, the optimal control is not constant. Then, by \Cref{lemma_noncostcontrol}, there exists an orthogonal matrix $M$, such that $u\circ M\neq u$. Now,
\begin{align}\label{inv_rotation_boundary}
I\left(u\circ M,z\right)&=\frac12\int_{\partial B(0,R)}|u\circ M|^2d\sigma(x)+\frac{\beta}{2}\int_{B(0,R)} |G(u\circ M)|^2dx\nonumber\\
&\;\hspace{0.33 cm}-\beta\int_{B(0,R)} G(u\circ M)zdx\nonumber\\
&=\frac12\int_{\partial B(0,R)}|u|^2d\sigma(x)+\frac{\beta}{2}\int_{B(0,R)} |G(u)|^2dx-\beta\int_{B(0,R)} G(u)zdx\\
&=I(u,z)\nonumber,
\end{align}
where in \cref{inv_rotation_boundary} we have employed the change of variable $\gamma(x)=Mx$ and \Cref{lemma_rot}. Then, $u$ and $u\circ M$ are two distinguished global minimizers for $I\left(\cdot,z\right)$, as desired. It remains to prove the nonuniqueness in case, for any radial targets, all the optimal controls are constants.\\
\textit{Step 2} \ \textbf{Existence of a special target $z^0\in L^{\infty}(B(0,R))$ such that $I(\cdot,z^0)$ admits (at least) two local minimizers among constant controls.}\\
By \Cref{lemma_rank_matrix_boundary}, there exists two controls $u_{-}<0<u_{+}$, such that \cref{def_Gamma_boundary} is invertible. We start proving the existence of a special target $z^0\in L^{\infty}(B(0,R))$ such that $I(u_{-},z^0)<0$ and $I(u_{+},z^0)<0$.


For an arbitrary target $z^0\in L^{\infty}(B(0,R))$, we have $I(u_{-},z^0)<0$ and $I(u_{+},z^0)<0$ if and only if the following system of inequalities is fulfilled:
\begin{equation}\label{th_nouniq_bound_eq6}
\begin{dcases}
\beta\mathlarger{\int_{B(0,R)}}G(u_{-})z^0 dx>\frac{R^{n-1}n\alpha(n)}{2}|u_{-}|^2+\frac{{\beta}}{2}\mathlarger{\int_{B(0,R)}}|G(u_{-})|^2dx 
\vspace{0.19 cm}\\
\beta\mathlarger{\int_{B(0,R)}}G(u_{+})z^0 dx>\frac{R^{n-1}n\alpha(n)}{2}|u_{+}|^2+\frac{{\beta}}{2}\mathlarger{\int_{B(0,R)}}|G(u_{+})|^2dx ,
\end{dcases}
\end{equation}
where $G$ is the control-to-state map introduced in \cref{control_to_state_PDE_boundary} and $\alpha(n)$ is the volume of the unit ball in $\mathbb{R}^n$. In the sequel, we work with \textit{changing-sign} targets
\begin{equation*}
z^0\coloneqq \begin{dcases}
z^0_1 \hspace{0.6 cm} &\mbox{in}\ \omega_1\\
z^0_2 \hspace{0.6 cm} &\mbox{in} \ \omega_2,
\end{dcases}
\end{equation*}
where $(z^0_1,z^0_2)\in \mathbb{R}^2$ and $\omega_1$ and $\omega_2$ are defined in \cref{lemma_molt_defomega1} and \eqref{lemma_molt_defomega2} respectively. $(z^0_1,z^0_2)$ are degrees of freedom we need in the remainder of the proof.
With the above choice of the target, inequalities \cref{th_nouniq_bound_eq6} are satisfied if the target $\left(z^0_1,z^0_2\right)$ satisfies the linear system below
\begin{equation}\label{th_nouniq_bound_eq9}
\begin{dcases}
z^0_1\beta\mathlarger{\int_{\omega_1}}G(u_{-})dx+z^0_2\beta\mathlarger{\int_{\omega_2}}G(u_{-})dx=c_1
\vspace{0.19 cm}\\
z^0_1\beta\mathlarger{\int_{\omega_1}}G(u_{+})dx+z^0_2\beta\mathlarger{\int_{\omega_2}}G(u_{+})dx=c_2,
\end{dcases}
\end{equation}
with constant terms
\begin{equation*}
c_1\coloneqq \frac{R^{n-1}n\alpha(n)}{2}|u_{-}|^2+\frac{{\beta}}{2}\mathlarger{\int_{B(0,R)}}|G(u_{-})|^2dx+1
\end{equation*}
and
\begin{equation*}
c_2\coloneqq \frac{R^{n-1}n\alpha(n)}{2}|u_{+}|^2+\frac{{\beta}}{2}\mathlarger{\int_{B(0,R)}}|G(u_{+})|^2dx +1.
\end{equation*}
The $2\times 2$ coefficients matrix for the above linear system reads as:
\begin{equation*}\label{def_Gamma_boundary_3}
\Gamma=\beta\begin{tikzpicture}[baseline=(current bounding box.center)]
\matrix (Gamma) [matrix of math nodes,nodes in empty cells,right delimiter={]},left delimiter={[} ]{
	\mathlarger{\int_{\omega_1}}G(u_{-})dx\hspace{0.3 cm}&\mathlarger{\int_{\omega_2}}G(u_{-})dx     \\
	\mathlarger{\int_{\omega_1}}G(u_{+})dx&\mathlarger{\int_{\omega_2}}G(u_{+}) dx  \\
} ;
\end{tikzpicture}
\end{equation*}

By \cref{def_Gamma_boundary}, the matrix $\Gamma$ is invertible. Therefore, by Rouch\'e-Capelli Theorem, there exists a solution to the linear system \cref{th_nouniq_bound_eq9}. Such solution $(z^0_1,z^0_2)$ defines a special target
\begin{equation*}
z^0\coloneqq \begin{dcases}
z^0_1 \hspace{0.6 cm} &\mbox{in}\ \omega_1\\
z^0_2 \hspace{0.6 cm} &\mbox{in} \ \omega_2,
\end{dcases}
\end{equation*}
such that $I(u_{-},z^0)<0$ and $I(u_{+},z^0)<0$.

We show now that $I\left(\cdot,z^0\right)$ admits (at least) two local minimizers. Indeed, by \Cref{lemma3_boundary} (1.), there exist:
\begin{equation*}
u_1\leq 0\hspace{ 0.5 cm}\mbox{such that}\hspace{0.5 cm}I(u_1,z^0)=\inf\left\{I(u,z) \ | \ u\equiv k, \ k\leq 0\right\}
\end{equation*}
and
\begin{equation*}
u_2\geq 0\hspace{0.5 cm}\mbox{such that}	\hspace{0.5 cm}I(u_2,z^0)=\inf\left\{I(u,z) \ | \ u\equiv k, \ k\geq 0\right\}.
\end{equation*}
Now,
\begin{equation*}
I(u_1,z^0)=\inf\left\{I(u,z) \ | \ u\equiv k, \ k\leq 0\right\}\leq I(u_{-},z^0)<0=I(0,z^0)
\end{equation*}
and
\begin{equation*}
I(u_2,z^0)=\inf\left\{I(u,z) \ | \ u\equiv k, \ k\geq 0\right\}\leq I(u_{+},z^0)<0=I(0,z^0).
\end{equation*}
Then, the control $u_1$ minimizes $I(\cdot,z^0)$ in the half line $\left(-\infty,0\right)$, while $u_2$ minimizes $I(\cdot,z^0)$ in the half line $\left(0,+\infty\right)$. We have found $u_1$ and $u_2$ two distinct local minimizers for $I(\cdot,z^0)$ in $\mathbb{R}$.\\
\textit{Step 3} \ \textbf{Conclusion}\\
We remind the definition of $h_1$ and $h_2$ given by \cref{def_h_1_PDE_boundary} and \cref{def_h_2_PDE_boundary} resp. In Step 2, we have determined $z^0\in L^{\infty}(B(0,R))$ such that $h_1(z^0)<0$ and $h_2(z^0)<0$. To finish our proof it suffices to find $\tilde{z}\in\mathbb{R}^n$ such that $h_1(\tilde{z})=h_2(\tilde{z})<0$. This follows from \Cref{lemma4_boundary}.
\end{proof}

\subsection{Internal control}
\label{subsec:nouniq.internal}

We introduce the well-known concept of radial control.
\begin{definition}\label{Def_radial_controlinternal}
A control $u:B(0,r)\longrightarrow \mathbb{R}$ is said to be radial if there exists $\psi:[0,r]\longrightarrow \mathbb{R}$, such that, for any $x\in B(0,r)$, we have $u(x)=\psi(\|x\|)$.
\end{definition}

Our strategy to prove \Cref{th_nouniq_int} resembles the one of \Cref{th_nouniq_bound}, except for Step 1, which consists now in a reduction to the radial controls instead of constant controls.

We define the control-to-state map
\begin{equation}\label{control_to_state_PDEinternal}
G:L^2(B(0,r))\longrightarrow L^2(B(0,R))
\end{equation}
\begin{equation*}
u\longmapsto y_u,
\end{equation*}
where $y_u$ is the solution to \cref{semilinear_internal_elliptic_1} with control $u$. Then, set:
\begin{equation}\label{def_functional_control_target_PDEinternal}
I:L^2(B(0,r))\times L^{\infty}(B(0,R)\setminus B(0,r))\longrightarrow \mathbb{R}
\end{equation}
\begin{equation*}
I(u,z)\coloneqq \frac12\int_{B(0,r)}|u|^2dx+\frac{\beta}{2}\int_{B(0,R)\setminus B(0,r)} |G(u)|^2dx-\beta\int_{B(0,R)\setminus B(0,r)} G(u)zdx,
\end{equation*}
where $G$ is the control-to-state map introduced in \cref{control_to_state_PDEinternal}. One recognizes that, for any target $z\in L^{\infty}(B(0,R)\setminus B(0,r))$, $I(\cdot,z)+\frac{{\beta}}{2} \|z\|_{L^2(B(0,R)\setminus B(0,r))}^2$ coincides with the functional $J$ defined in \cref{functionalinternal} with target $z$. Then, for any target $z\in L^{\infty}(B(0,R)\setminus B(0,r))$ minimizing $I(\cdot, z)$ is equivalent to minimizing $J$ with target $z$. Such translation is convenient, because $I(0,z)=0$ for any target $z\in L^{\infty}(B(0,R)\setminus B(0,r))$.

We establish some important properties of the solutions of the state equation \cref{semilinear_internal_elliptic_1}:
\begin{itemize}
\item The unique constant solution of the equation $-\Delta y+f(y) = 0$ in $B(0,R)$, with $y=0$ on $\partial B(0,R)$ is $y \equiv 0$ (\Cref{lemma_nonconst_internal}). In particular, $G(u) = 0$ if and only if $u = 0$ holds.
\item By comparison principle,
if $u \geq 0$ in $B(0,r)$ and $u \not\equiv 0$, then $G(u)(x) > 0$ in $B(0,R)$.
\item By comparison principle,
if $u \leq 0$ in $B(0,r)$ and $u \not\equiv 0$, then $G(u)(x) < 0$ in $B(0,R)$.
\end{itemize}

We define
\begin{equation}\label{Def_U_r}
\mathscr{U}_r\coloneqq \left\{u\in L^2\left({B(0,r)}\right) \ | \ \mbox{$u$ is radial} \right\}.
\end{equation}
We have
\begin{equation}
\mathscr{U}_r = \mathscr{U}_r^- \cup \mathscr{U}_r^-,
\end{equation}
with
\begin{align}
\mathscr{U}_r^-&\coloneqq \left\{u\in \mathscr{U}_r \ \big| \ G(u)\hspace{-0.1 cm}\restriction_{\partial B(0,r)}\leq 0\right\}\nonumber\\
\mathscr{U}_r^+&\coloneqq \left\{u\in \mathscr{U}_r \ \big| \ G(u)\hspace{-0.1 cm}\restriction_{\partial B(0,r)}\geq  0\right\}.
\end{align}

We introduce:
\begin{equation}\label{def_h_1_PDEinternal}
h_1:L^{\infty}(B(0,R)\setminus B(0,r))\longrightarrow \mathbb{R},\hspace{0.3 cm}  h_1(z)\coloneqq\inf\left\{I(u,z) \ | \ u\in \mathscr{U}_r^-\right\}
\end{equation}
and
\begin{equation}\label{def_h_2_PDEinternal}
h_2:L^{\infty}(B(0,R)\setminus B(0,r))\longrightarrow \mathbb{R},\hspace{0.3 cm} h_2(z)\coloneqq\inf\left\{I(u,z) \ | \ u\in \mathscr{U}_r^+\right\}.
\end{equation}

We formulate the first Lemma.

\begin{lemma}\label{lemma3_internal}
Let $C= \mathscr{U}_r^-$ or $C= \mathscr{U}_r^+$. Then,
\begin{enumerate}
	\item for any $z\in L^{\infty}(B(0,R)\setminus B(0,r))$, there exists $u_{z}\in C$ such that:
	\begin{equation*}
	I(u_{z},z)=\inf_{C}[I(\cdot,z)].
	\end{equation*}
	Furthermore, for any minimizer $u_z$
	\begin{equation*}
	\|u_{z}\|_{L^2(B(0,r))}\leq \sqrt{\beta}\|z\|_{L^2}.
	\end{equation*}
	\item the map
	\begin{equation*}
	h:L^{\infty}(B(0,R)\setminus B(0,r))\longrightarrow \mathbb{R}
	\end{equation*}
	\begin{equation*}
	z\longmapsto \inf_{C}\left[I(\cdot, z)\right]
	\end{equation*}
	is continuous.
\end{enumerate}
\end{lemma}
The proof of \Cref{lemma3_internal} resembles the one of \Cref{lemma3_boundary}, available in \cref{sec:appendix.Preliminaries for boundary control}. We now state the second lemma needed to prove \Cref{th_nouniq_int}.

\begin{lemma}\label{lemma4_internal}
Assume there exists $z^0\in L^{\infty}(B(0,R)\setminus B(0,r))$ such that
\begin{equation*}
h_1(z^0)<0\hspace{1 cm}\mbox{and}\hspace{1 cm}h_2(z^0)<0,
\end{equation*}
where $h_1$ and $h_2$ are defined in \cref{def_h_1_PDEinternal} and \cref{def_h_2_PDEinternal} resp.
Then, there exists $\tilde{z}\in L^{\infty}(B(0,R)\setminus B(0,r))$ such that
\begin{equation*}
h_1(\tilde{z})=h_2(\tilde{z})<0.
\end{equation*}
\end{lemma}

The above Lemma can be proved by following the arguments of \Cref{lemma4_boundary}, in \cref{sec:appendix.Preliminaries for boundary control}. The next lemma is the foundation of the proof of the existence of two local minimizers for \cref{functionalinternal}. The nonlinearity of the state equation \cref{semilinear_internal_elliptic_1} will play a key role in the proof.

\begin{lemma}\label{lemma_rank_matrix_internal}
Let $\Omega$ be a bounded open subset of $\mathbb{R}^n$, with $\partial \Omega \in C^{\infty}$ and $\omega \subsetneq \Omega$ a nonempty open subset. Let $u_{-}<0<u_{+,1}<u_{+,2}$ be three constant controls. For any $u\in L^{2}\left(\omega\right)$, let $G\left(u\right)$ be the solution to
\begin{equation}\label{semilinear_internal_elliptic_1_gdomain_wp_2}
\begin{dcases}
-\Delta y+f(y)=u\chi_{\omega}\hspace{2.8 cm} & \mbox{in} \hspace{0.10 cm}\Omega\\
y=0  & \mbox{on}\hspace{0.10 cm} \partial \Omega.
\end{dcases}
\end{equation}
Assume $f\in C^1\left(\mathbb{R}\right)\cap C^2\left(\mathbb{R}\setminus \left\{0\right\}\right)$ is strictly increasing and
\begin{equation*}
f^{\prime\prime}(y)\neq 0\hspace{1 cm}\forall \ y\neq 0.
\end{equation*}
Set
\begin{equation}
\overline{\lambda}\coloneqq \frac{\int_{\Omega}G\left(u_{+,2}\right)(x)dx}{\int_{\Omega}G\left(u_{+,1}\right)(x)dx},
\end{equation}
\begin{equation}\label{lemma_molt_defomega1_internal}
\omega_1\coloneqq \left\{x\in \Omega\setminus \omega \ | \ G\left(u_{+,2}\right)(x)<\overline{\lambda} \hspace{0.003 cm} G\left(u_{+,1}\right)(x)\right\}
\end{equation}
and
\begin{equation}\label{lemma_molt_defomega2_internal}
\omega_2\coloneqq \left\{x\in \Omega\setminus \omega \ | \ G\left(u_{+,2}\right)(x)>\overline{\lambda} \hspace{0.003 cm} G\left(u_{+,1}\right)(x)\right\}.
\end{equation}
There exist $i\in \left\{1,2\right\}$, such that
\begin{equation}\label{def_Gamma_internal}
\Gamma\coloneqq \beta\begin{tikzpicture}[baseline=(current bounding box.center)]
\matrix (Gamma) [matrix of math nodes,nodes in empty cells,right delimiter={]},left delimiter={[} ]{
	\mathlarger{\int_{\omega_1}}G(u_{-})dx\hspace{0.3 cm}&\mathlarger{\int_{\omega_2}}G(u_{-})dx     \\
	\mathlarger{\int_{\omega_1}}G(u_{+,i})dx&\mathlarger{\int_{\omega_2}}G(u_{+,i}) dx  \\
} ;
\end{tikzpicture}.
\end{equation}
is invertible.
\end{lemma}
The proof of the above Lemma
resembles the one of \Cref{lemma_rank_matrix_boundary}. A key point is that, being in the complement of the control region, for $i=1,2$, we have
\begin{equation}
-\Delta G\left(u_{+,i}\right) + f\left( G\left(u_{+,i}\right)\right)=0\hspace{0.6 cm} \mbox{in} \hspace{0.10 cm} \Omega\setminus \omega.
\end{equation}

\begin{proof}[Proof of \Cref{th_nouniq_int}.]
\textit{Step 1} \ \textbf{Reduction to radial controls.}\\
Suppose for some radial target $z$, the optimal control $u$ is not radial, that is there exists an orthogonal matrix $M$, such that $u\circ M\neq u$. By \Cref{lemma_rot_int}, we have $G\left(u\circ M\right)=G\left(u\right)\circ M$. Now,
\begin{align}\label{inv_rotationinternal}
I\left(u\circ M,z\right)&=\frac12\int_{B(0,r)}|u\circ M|^2dx+\frac{\beta}{2}\int_{B(0,R)\setminus B(0,r)} |G(u\circ M)|^2dx\nonumber\\
&\;\hspace{0.33 cm}-\beta\int_{B(0,R)\setminus B(0,r)} G(u\circ M)zdx\nonumber\\
&=\frac12\int_{B(0,r)}|u|^2dx+\frac{\beta}{2}\int_{B(0,R)\setminus B(0,r)} |G(u)|^2dx\\
&\;\hspace{0.33 cm}-\beta\int_{B(0,R)\setminus B(0,r)} G(u)zdx\nonumber\\
&=I(u,z)\nonumber,
\end{align}
where in the last equality \cref{inv_rotationinternal} we have employed the change of variable $\gamma(x)=Mx$. Then, $u$ and $u\circ M$ are two distinguished global minimizers for $I\left(\cdot,z\right)$, as desired. It remains to prove the nonuniqueness in case, for any radial target, all the optimal controls are radial. Hereafter, for a radial target $z$, we will consider the restriction of the functional $I(\cdot, z)$ to $\mathscr{U}_r$.\\
\textit{Step 2} \ \textbf{Existence of a special target $z^0\in L^{\infty}(B(0,R)\setminus B(0,r))$ such that $I(\cdot,z^0)$ admits (at least) two local minimizers, among radial controls.}\\
By \Cref{lemma_rank_matrix_internal}, there exists two controls $u_{-}<0<u_{+}$, such that \cref{def_Gamma_internal} is invertible.
Proceeding as in Step 2 of the proof of \Cref{th_nouniq_bound}, one can prove the existence of a special target
\begin{equation*}
z^0\coloneqq \begin{dcases}
z^0_1 \hspace{0.6 cm} &\mbox{in}\ \omega_1\\
z^0_2 \hspace{0.6 cm} &\mbox{in} \ \omega_2
\end{dcases}
\end{equation*}
such that $I(u_{-},z^0)<0$ and $I(u_{+},z^0)<0$. Note that in this case $\omega_1$ and $\omega_2$ are defined in \cref{lemma_molt_defomega1_internal} and \cref{lemma_molt_defomega2_internal} respectively.

We show now that $I\left(\cdot,z^0\right)$ admits (at least) two local minimizers in $\mathscr{U}_r$. Indeed, the set $\mathscr{U}_r$ (introduced in \cref{Def_U_r}) splits
\begin{equation*}
\mathscr{U}_r = \mathscr{U}_r^- \cup \mathscr{U}_r^+,
\end{equation*}
with
\begin{align}
\mathscr{U}_r^-&= \left\{u\in \mathscr{U}_r \ \big| \ G(u)\hspace{-0.1 cm}\restriction_{\partial B(0,r)}\leq 0\right\}\nonumber\\
\mathscr{U}_r^+&= \left\{u\in \mathscr{U}_r \ \big| \ G(u)\hspace{-0.1 cm}\restriction_{\partial B(0,r)}\geq 0\right\},
\end{align}
where we have used that for any radial control $u$, by \Cref{lemma_rot_int}, $G(u)$ is radial and (by elliptic regularity \cite[Theorem 4 page 334]{PDE}) continuous, so that $G(u)\hspace{-0.1 cm}\restriction_{\partial B(0,r)}$ is a real number.

By \Cref{lemma3_internal} (1.), there exist:
\begin{equation*}
u_1\in \mathscr{U}_r^-\hspace{ 0.5 cm}\mbox{such that}\hspace{0.5 cm}I(u_1,z^0)=\inf_{\mathscr{U}_r^-}[I(\cdot,z^0)]
\end{equation*}
and
\begin{equation*}
u_2\in \mathscr{U}_r^+\hspace{0.5 cm}\mbox{such that}	\hspace{0.5 cm}I(u_2,z^0)=\inf_{\mathscr{U}_r^+}[I(\cdot,z^0)].
\end{equation*}
Now, for any control $u\in \left\{u\in \mathscr{U}_r \ \big| \ G(u)\hspace{-0.1 cm}\restriction_{\partial B(0,r)}= 0\right\}$,
we have
\begin{equation*}
I(u_1,z^0)=\inf_{\mathscr{U}_r^-}[I(\cdot,z^0)]\leq I(u_{-},z^0)<0\leq I(u,z^0)
\end{equation*}
and
\begin{equation*}
I(u_2,z^0)=\inf_{\mathscr{U}_r^+}[I(\cdot,z^0)]\leq I(u_{+},z^0)<0\leq I(u,z^0).
\end{equation*}
Then, necessarily $u_1$ is a local minimizer for $I(\cdot,z^0)$ in the open set\\
$\left\{u\in \mathscr{U}_r \ \big| \ G(u)\hspace{-0.1 cm}\restriction_{\partial B(0,r)}< 0\right\}$ and $u_2$ is a local minimizer for $I(\cdot,z^0)$ in the open set\\
$\left\{u\in \mathscr{U}_r \ \big| \ G(u)\hspace{-0.1 cm}\restriction_{\partial B(0,r)}> 0\right\}$. Hence, we have found $u_1$ and $u_2$ two distinct local minimizers for $I(\cdot, z^0)$ in $\mathscr{U}_r$.\\
\textit{Step 3} \ \textbf{Conclusion}\\
We remind the definition of $h_1$ and $h_2$ given by \cref{def_h_1_PDEinternal} and \cref{def_h_2_PDEinternal} resp. In Step 2, we have determined $z^0\in L^{\infty}(B(0,R)\setminus B(0,r))$ such that $h_1(z^0)<0$ and $h_2(z^0)<0$. To finish our proof it suffices to find $\tilde{z}\in\mathbb{R}^n$ such that $h_1(\tilde{z})=h_2(\tilde{z})<0$. This follows from \Cref{lemma4_internal}.
\end{proof}

\section{Numerical simulations}
\label{sec:simulations}

We have performed a numerical simulation in the context of boundary control. We illustrate in \cref{fig:functionalnouniq3} an example, with step target
\begin{equation}\label{step_target}
z(x)\coloneqq \begin{dcases}
410000 \hspace{0.6 cm} &\mbox{for}\ 0<x<\frac14 \mbox{ and } \frac34<x<1\\
-10300000 \hspace{0.6 cm} &\mbox{for} \ \frac14<x<\frac34.
\end{dcases}
\end{equation}

\begin{figure}[tbhp]
\begin{center}
	\centering
	\includegraphics[width=13 cm]{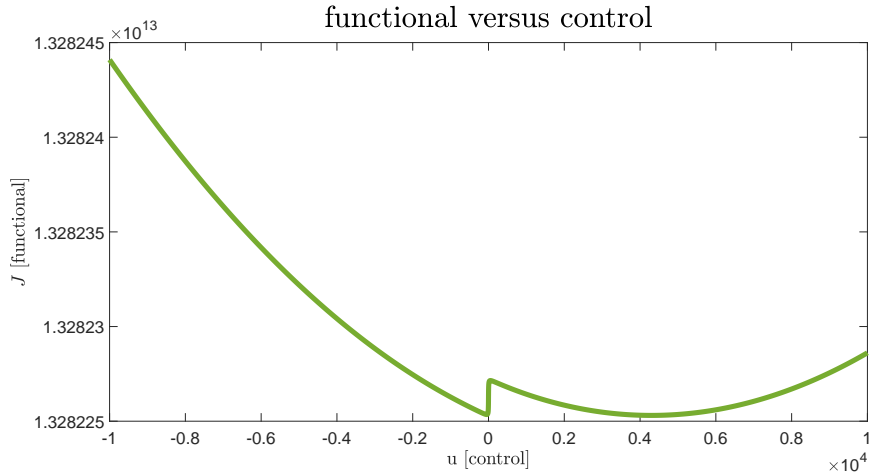}
	\caption{functional versus control (nonuniqueness of the \textit{global} minimizer). This plot is obtained by drawing in MATLAB the graph of $J$ defined in \cref{functional_nouniqboundary}, with space dimension $n=1$, $R=1$, weighting parameter $\beta=1$ and target \cref{step_target}.}
	\label{fig:functionalnouniq3}
\end{center}
\end{figure}

As we have seen in the proof of \Cref{th_nouniq_bound}, we can reduce to the case of constant controls on the boundary. In our case, the space dimension is $n=1$. Then, we have reduced to the case the same control acts on both endpoints $x=0$ and $x=1$. Hence, we plot in \cref{fig:functionalnouniq3} the restriction $J\hspace{-0.1 cm}\restriction_{\mathbb{R}}:\mathbb{R}\longrightarrow \mathbb{R}$, the functional $J$ being defined in \cref{functional_nouniqboundary}.

There exist two distinguished global minimizers:
\begin{itemize}
\item a negative one $u_1\cong -50$;
\item a positive one $u_2\cong 4298$.
\end{itemize}
The corresponding optimal states are depicted in figures \cref{stateminus50} and \cref{state4298}.

\begin{figure}[tbhp]
\begin{center}
	\centering
	\includegraphics[width=12 cm]{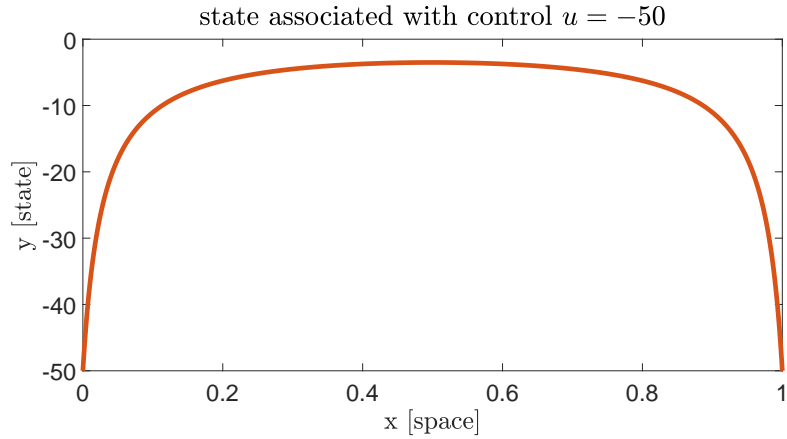}
	\caption{state associated with control $u=-50$.}\label{stateminus50}
\end{center}
\end{figure}
\begin{figure}[tbhp]
\begin{center}
	\centering
	\includegraphics[width=12 cm]{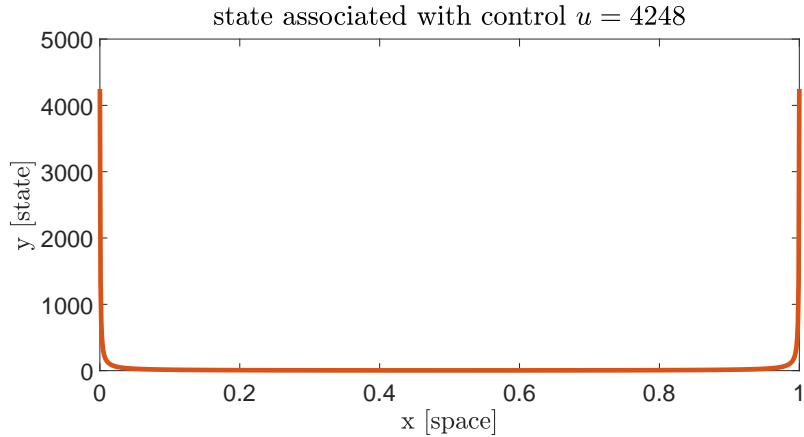}
	\caption{state associated with control $u=4298$.}\label{state4298}
\end{center}
\end{figure}


The idea behind this example is that two optimal strategies are available:
\begin{itemize}
\item take a large positive control $u_2$ to better approximate the target in $\left(0,\frac14\right)\cup\left(\frac34,1\right)$;
\item take a negative control $u_1$ to keep the state closer to the target in $\left(\frac14,\frac34\right)$.
\end{itemize}
Note that $|u_1|<|u_2|$. Indeed, the control acts at the endpoints $x=0$ and $x=1$ of the space domain. Then, the effect of the control is stronger in $\left(0,\frac14\right)\cup\left(\frac34,1\right)$ than in $\left(\frac14, \frac34\right)$. For this reason, it is worth to take a large positive control to better approximate the target in $\left(0,\frac14\right)\cup\left(\frac34,1\right)$. On the other hand, it is less convenient to take a very negative control to approximate the target in $\left(\frac14, \frac34\right)$ (see the local estimates for semilinear equations \cite{henry1977etude} and \cite[proof of Theorem 1.3]{EFR}).

In \cref{fig:functionalnouniq3} we observe that the functional has a different behaviour close to zero and away from zero. This can be explained by studying the behaviour of the control-to-state map \cref{control_to_state_PDE_boundary}:
\begin{itemize}
\item close to zero \cref{control_to_state_PDE_boundary} is closed to its linearization around zero;
\item far from zero \cref{control_to_state_PDE_boundary} is strongly influenced by the nonlinearity $f(y)=y^3$, thus producing a drastic change in the shape of the functional.
\end{itemize}


Numerical simulations have been performed in MATLAB. We explain now the numerical methods employed.

Firstly choose an interval of controls $[-M,M]$, where to study the functional $J$. Then, our goal is to plot $J\hspace{-0.1 cm}\restriction_{[-M,M]}:[-M,M]\longrightarrow \mathbb{R}$.

For the interval $[-M,M]$, we choose an equi-spaced grid $v_i=-M+(i-1)\frac{2M}{N_c-1}$, with $i=1,\dots ,N_c$ and $N_c\in \mathbb{N}\setminus \left\{0\right\}$.


Now, for each control $v_i$, we need to find numerically the corresponding state $y_i$, solution to the following PDE with cubic nonlinearity
\begin{equation}\label{semilinear_internal_elliptic_1_one_dim}
\begin{dcases}
-\left(y_i\right)_{xx}+(y_i)^3=0\hspace{2.8 cm} & \hspace{0.10 cm}x\in (0,1)\\
y_i(0)=y_i(1)=v_i.
\end{dcases}
\end{equation}

Following \cite[subsubsection 4.3.2]{boyer2017controllability}, we solve \cref{semilinear_internal_elliptic_1_one_dim} by a fixed-point type algorithm with relaxation. Namely, in any iteration $k$, we determine the solution $y_{i,k}$ to the linear PDE
\begin{equation}\label{semilinear_internal_elliptic_1_one_dim_linearized}
\begin{dcases}
-(y_{i,k})_{xx}+(\theta_{i,k-1})^2y_{i,k}=0\hspace{2.8 cm} & \hspace{0.10 cm}x\in (0,1)\\
y_{i,k}(0)=y_{i,k}(1)=v_i
\end{dcases}
\end{equation}
and we set $\theta_k\coloneqq \frac12 \theta_{i,k-1}+\frac12 y_k$. The initial guess $\theta_{i,0}$ is taken to be $y_{i-1}$, i.e. the solution to \cref{semilinear_internal_elliptic_1_one_dim}, with control $v_{i-1}$.

To compute the solution to the linear PDE \cref{semilinear_internal_elliptic_1_one_dim_linearized}, we choose a finite difference scheme with uniform space grid $x_j=\frac{j-1}{\Delta x}$, where $j=1,\dots,N_x$, $N_x\in \mathbb{N}\setminus \left\{0\right\}$ and $\Delta x\coloneqq \frac{1}{N_x-1}$. Then, $y_{i,k}=\left(y_{i,k,j}\right)_{j}$ is a $N_x$-dimensional discrete vector solution to
\begin{equation*}\label{semilinear_internal_elliptic_1_one_dim_linearized_finite_difference}
\begin{dcases}
\frac{-y_{i,k,j-1}+2y_{i,k,j}-y_{i,k,j+1}}{(\Delta x)^2}+(\theta_{i,k-1,j})^2y_{i,k,j}=0\hspace{1.8 cm} & \hspace{0.10 cm}j=2,\dots,N_x-1\\
y_{i,k,1}=y_{i,k,N_x}=v_i.
\end{dcases}
\end{equation*}

%
Once we have determined the state $y_i$, we evaluate the functional $J$ at the control $v_i$. The integral appearing in \cref{functional_nouniqboundary} can be computed by quadrature methods.
%
%
We are now in position to plot the functional $J\hspace{-0.1 cm}\restriction_{[-M,M]}:[-M,M]\longrightarrow \mathbb{R}$.

Note that, as long as we know, the actual convergence of the fixed-point method described has not been proved. However, for any control $v_i$, we are able to check that the state computed solves the finite difference version of the nonlinear problem \cref{semilinear_internal_elliptic_1_one_dim} up to a small error.

An extensive literature is available on the numerical approximation of solutions to \cref{semilinear_internal_elliptic_1_one_dim} (see, for instance, \cite{glowinski2013numerical} for a survey). Let us mention two alternative numerical methods.\\
The first one is a finite difference-Newton method presented in \cite[subsection 2.16.1]{leveque2007finite}. The idea is to discretize directly \cref{semilinear_internal_elliptic_1_one_dim}. This leads to a nonlinear equation in finite dimension, solved by a Newton method.\\
Another option is to find the solution to \cref{semilinear_internal_elliptic_1_one_dim}, as minimizer of the convex functional
\begin{equation*}
K(y)=\frac12\int_0^1 |y_x|^2 dx+\frac14 \int_0^1 y^4 dx
\end{equation*}
over the affine space
\begin{equation*}
\mathcal{A}\coloneqq \left\{y\in H^1(0,1) \ | \ y(0)=y(1)=v\right\}.
\end{equation*}


\section{Conclusions and open problems}
\label{sec:conclusions}

We have illustrated a general methodology to show lack of convexity for quadratic functionals with nonlinear state equations (\Cref{thlackconv}). Furthermore, we have developed a counterexample to uniqueness of the global minimizer in optimal control of semilinear elliptic equations (\Cref{th_nouniq_bound} and \Cref{th_nouniq_int}).

We list some interesting problems, which, to the best of our knowledge, have not been addressed in the literature so far.

\subsection{General space domain}
\label{subsec:conclusionsGeneral space domain}

Our counterexample to uniqueness of the minimizer in semilinear control relies on the rotational invariance of the space domain $B(0,R)$ to reduce to constant/radial controls. It would be interesting to enhance the developed techniques to more general space domains.

\subsection{Relations with the turnpike property}
\label{subsec:conclusions}

Consider the time-evolution control problem associated to \cref{semilinear_internal_elliptic_1}-\cref{functionalinternal}
\begin{equation}\label{functional_turnpike_op}
\min_{u\in \mathscr{U}_T}J_{T}(u)=\frac12\int_0^T\int_{B(0,r)} |u|^2 dxdt+ \frac{\beta}{2}\int_0^T\int_{B(0,R)\setminus B(0,r)} |y-z|^2 dxdt,
\end{equation}
where $\mathscr{U}_T\coloneqq L^2((0,T)\times B(0,r))$ and the state $y$ associated to control $u$ is solution to the semilinear heat equation
\begin{equation}\label{semilinear_internal_1_turnpike_op}
\begin{dcases}
y_t-\Delta y+f(y)=u\chi_{B(0,r)}\hspace{2.8 cm} & \mbox{in} \hspace{0.10 cm}(0,T)\times B(0,R)\\
y=0  & \mbox{on}\hspace{0.10 cm} (0,T)\times \partial B(0,R)\\
y(0,x)=y_0(x)  & \mbox{in}\hspace{0.10 cm}  B(0,R).
\end{dcases}
\end{equation}
The nonlinearity $f$ is $C^3$ and nondecreasing, with $f(0)=0$. The assumptions on the state equation are the same of \cite[section 3]{PZ2}. An optimal control for the above problem is denoted by $u^T$, while the corresponding optimal state by $y^T$.

We rewrite \cref{semilinear_internal_elliptic_1}-\cref{functionalinternal} with an ``s'' subscript to stress the steady-state character of the problem
\begin{equation}\label{steady_functional_turnpike_op}
\min_{u_s\in L^2(B(0,r))}J_s(u_s)=\frac12\int_{B(0,r)} |u_s|^2 dx +\frac{\beta}{2}\int_{B(0,R)\setminus B(0,r)} |y_s-z|^2 dx,
\end{equation}
where:
\begin{equation}\label{semilinear_internal_elliptic_1_turnpike_op}
\begin{dcases}
-\Delta y_s+f\left(y_s\right)=u_s\chi_{B(0,r)}\hspace{2.8 cm} & \mbox{in} \hspace{0.10 cm}B(0,R)\\
y_s=0  & \mbox{on}\hspace{0.10 cm} \partial B(0,R).
\end{dcases}
\end{equation}
We denote by $(\overline{u},\overline{y})$ an optimal pair, where $\overline{u}$ is an optimal control and $\overline{y}$ the corresponding optimal state.

Consider a target $z$, such that $J_s$ has two distinguished global minimizers, as in \Cref{th_nouniq_int}. Choose any initial datum $y_0\in L^{\infty}(B(0,R))$ for the evolution equation \cref{semilinear_internal_1_turnpike_op}. Let $u^T$ be a minimizer for \eqref{functional_turnpike_op}. Then, a question arises: if the turnpike property is satisfied, which minimizer for \cref{semilinear_internal_elliptic_1_turnpike_op}-\cref{steady_functional_turnpike_op} attracts the optimal solutions to \cref{semilinear_internal_1_turnpike_op}-\cref{functional_turnpike_op}? Namely, for which optimal pair $\left(\overline{u},\overline{y}\right)$ for \cref{semilinear_internal_elliptic_1_turnpike_op}-\cref{steady_functional_turnpike_op} we have the estimate
\begin{equation*}
\|u^T(t)-\overline{u}\|_{L^{\infty}(B(0,r))}+\|y^T(t)-\overline{y}\|_{L^{\infty}(B(0,R))}\leq K\left[e^{-\mu t}+e^{-\mu(T-t)}\right],\hspace{0.6 cm}\forall t\in [0,T],
\end{equation*}
where the constants $K$ and $\mu>0$ are independent of the time horizon $T$.

According to \cite[Theorem 1, section 3]{PZ2}, this depends on the sign of the second differential of the functional $J_s$ computed at the minima, which in turns is linked to the sign of the term $\beta\chi_{B(0,R)\setminus B(0,r)}-f^{\prime\prime}\left(\overline{y}\right)\overline{q}$.

\appendix

\section{Preliminaries for boundary control}
\label{sec:appendix.Preliminaries for boundary control}

In this section, we present some results in boundary control. We accomplish this task in a general space domain $\Omega$.

Let $\Omega$ be a bounded open subset of $\mathbb{R}^n$, with boundary $\partial \Omega \in C^{\infty}$. The nonlinearity $f\in C^1(\mathbb{R})$ is increasing and $f(0)=0$. We introduce the class of test functions
\begin{equation*}
\mathscr{C}\coloneqq \left\{\varphi\in C^2\left(\overline{\Omega}\right) \ | \ \varphi(x)=0, \ \forall x\in \partial \Omega\right\}
\end{equation*}
and the notion of solution.

\begin{definition}\label{def_sol_semilinear_boundary_elliptic}
Let $u\in L^{\infty}(\partial \Omega)$. Then, $y\in L^{\infty}(\Omega)$ is said to be a solution to the boundary value problem
\begin{equation}\label{semilinear_boundary_elliptic_1_gdomain_wp}
\begin{dcases}
-\Delta y+f(y)=0\hspace{2.8 cm} & \mbox{in} \hspace{0.10 cm}\Omega\\
y=u  & \mbox{on}\hspace{0.10 cm} \partial \Omega.
\end{dcases}
\end{equation}
if for any test function $\varphi\in \mathscr{C}$, we have
\begin{equation*}
\int_{\Omega}\left[-y\Delta \varphi+f(y)\varphi\right]dx+\int_{\partial \Omega}u\frac{\partial \varphi}{\partial n} d\sigma(x)=0,
\end{equation*}
where $n$ is the outward normal to $\partial \Omega$.
\end{definition}

We have the following existence and uniqueness result, inspired by the proof of \cite[Proposition 5.1]{ACS}.

\begin{proposition}\label{prop_exuniqellipticsemilinear}
Let $u\in L^{\infty}(\partial \Omega)$. There exists a unique solution\\
$y\in L^{\infty}(\Omega)\cap H^{\frac12}(\Omega)$ to \eqref{semilinear_boundary_elliptic_1_gdomain_wp}, with estimate
\begin{equation}\label{estimate_L^2*}
\left\|y\right\|_{L^{2^*}(\Omega)}\leq K\left\|u\right\|_{L^2(\partial \Omega)},
\end{equation}
the constant $K=K(\Omega)$ being independent of the nonlinearity $f$ and $2^*=\frac{2n}{n-1}$. If the boundary control $u\in H^{\frac12}\left(\partial\Omega\right)\cap C^0\left(\partial \Omega\right)$, then in fact $y\in H^1\left(\Omega\right)\cap C^0\left(\overline{\Omega}\right)$.
\end{proposition}
One of the key points of the proof will be the increasing character of the nonlinearity.
\begin{proof}[Proof of \Cref{prop_exuniqellipticsemilinear}]
\textit{Step 1} \ \textbf{Solve a non-homogeneous linear problem}\\
By \cite[Th\'eor\`eme 7.4, page 202]{LM1}, there exists a unique solution $y_1\in H^{\frac12}(\Omega)$ to the non-homogeneous boundary value problem
\begin{equation}\label{linear_boundary_elliptic_1_gdomain_wp}
\begin{dcases}
-\Delta y_1=0\hspace{2.8 cm} & \mbox{in} \hspace{0.10 cm}\Omega\\
y_1=u  & \mbox{on}\hspace{0.10 cm} \partial \Omega.
\end{dcases}
\end{equation}
The boundary value $u\in L^{\infty}(\partial \Omega)$. Hence, by a comparison argument, we have $y_1\in L^{\infty}(\Omega)$.\\
\textit{Step 2} \ \textbf{Solve an homogeneous semilinear problem}\\
Since the nonlinearity $f$ is increasing, by adapting the techniques of \cite[Theorem 4.7, page 29]{boccardo2013elliptic}, there exists a unique $y_2\in H^1_0(\Omega)$ solution to
\begin{equation}\label{semilinear_boundary_elliptic_1_gdomain_wp_y2}
\begin{dcases}
-\Delta y_2+f\left(y_1+y_2\right)=0\hspace{2.8 cm} & \mbox{in} \hspace{0.10 cm}\Omega\\
y_2=0  & \mbox{on}\hspace{0.10 cm} \partial \Omega.
\end{dcases}
\end{equation}
By a comparison argument, since $y_1\in L^{\infty}(\Omega)$, we have $y_2\in L^{\infty}(\Omega)$. Then, $y=y_1+y_2\in L^{\infty}(\Omega)\cap H^{\frac12}(\Omega)$ is the unique solution to \cref{semilinear_boundary_elliptic_1_gdomain_wp}.\\
\textit{Step 3} \  \textbf{Proof of the estimate \cref{estimate_L^2*}}\\
By a comparison argument, we have
\begin{equation}\label{comparison_state_wp}
\left|y\right|\leq \hat{y}, \hspace{0.3 cm}\mbox{a.e.} \ \Omega,
\end{equation}
with
\begin{equation}\label{upper_semilinear_boundary_elliptic_1_gdomain_constrained_wp}
\begin{dcases}
-\Delta \hat{y}=0\hspace{2.8 cm} & \mbox{in} \hspace{0.10 cm}\Omega\\
\hat{y}=\left|u\right|  & \mbox{on}\hspace{0.10 cm} \partial \Omega.
\end{dcases}
\end{equation}
Now, by \cite[Th\'eor\`eme 7.4, page 202]{LM1}, the solution $\hat{y}\in H^{\frac12}(\Omega)$, with estimate
\begin{equation}\label{est_barrsol_lin}
\left\|\hat{y}\right\|_{H^{\frac12}(\Omega)}\leq K\left\|u\right\|_{L^2(\partial \Omega)}.
\end{equation}
The above inequality, together with the fractional Sobolev embedding $H^{\frac12}(\Omega)\hookrightarrow L^{2^*}(\Omega)$ (see e.g. \cite[Theorem 6.7]{HFV}), yields
\begin{equation*}
\left\|\hat{y}\right\|_{L^{2^*}(\Omega)}\leq\left\|\hat{y}\right\|_{H^{\frac12}(\Omega)}\leq K\left\|u\right\|_{L^2(\partial \Omega)},
\end{equation*}
whence by \cref{comparison_state_wp}, we have
\begin{equation*}
\left\|y\right\|_{L^{2^*}(\Omega)}\leq \left\|\hat{y}\right\|_{L^{2^*}(\Omega)}\leq K\left\|u\right\|_{L^2(\partial \Omega)},
\end{equation*}
with $K=K(\Omega)$, as required.\\
\textit{Step 4} \  \textbf{Improved regularity}\\
Since $\partial \Omega\in C^{\infty}$, by \cite[Th\'eor\`eme 7.4, page 202]{LM1} and \cite[Proposition 1.29 page 14]{giaquinta2013introduction}, the solution to \cref{linear_boundary_elliptic_1_gdomain_wp} $y_1\in H^1\left(\Omega\right)\cap C^0\left(\overline{\Omega}\right)$. Now, $y_2$ solves the linear problem
\begin{equation}\label{semilinear_boundary_elliptic_1_gdomain_wp_y2_linear}
\begin{dcases}
-\Delta y_2+cy_2=-f\left(y_1\right)\hspace{2.8 cm} & \mbox{in} \hspace{0.10 cm}\Omega\\
y_2=0  & \mbox{on}\hspace{0.10 cm} \partial \Omega.
\end{dcases}
\end{equation}
with bounded coefficient
\begin{equation*}c(x)\coloneqq
\begin{cases}
\frac{f\left(y_1(x)+y_2(x)\right)-f\left(y_1(x)\right)}{y_2(x)} \hspace{1.3 cm} & y_2(x)\neq 0\\
f^{\prime}(y_1(x)) & y_2(x)= 0.\\
\end{cases}
\end{equation*}
Then, by \cite[Th\'eor\`eme 7.4, page 202]{LM1} and \cite[Theorem 8.30 page 206]{EPG} applied to \eqref{semilinear_boundary_elliptic_1_gdomain_wp_y2_linear}, $y_2\in H^1\left(\Omega\right)\cap C^0\left(\overline{\Omega}\right)$. Hence, $y=y_1+y_2\in H^1\left(\Omega\right)\cap C^0\left(\overline{\Omega}\right)$, as desired.
\end{proof}

We now state and prove some Lemmas needed in the manuscript.



\begin{lemma}\label{lemma_rangesolutionelliptic}
Let $u\in L^{\infty}(\partial \Omega)$ be a control. Let $y$ be the solution to \eqref{semilinear_boundary_elliptic_1_gdomain_wp}, with control $u$. Assume the nonlinearity $f$ is strictly increasing and $y$ is constant. Then, $y\equiv 0$ and $u\equiv 0$.
\end{lemma}
\begin{proof}[Proof of \Cref{lemma_rangesolutionelliptic}]
Suppose there exists $c\in\mathbb{R}$, such that $y(x)=c$, for any $x\in \Omega$. Then, by \Cref{def_sol_semilinear_boundary_elliptic}, for any for any test function $\varphi\in C^{\infty}_c(\Omega)$, we have
\begin{equation*}
\int_{\Omega}\left[-c\Delta \varphi+f(c)\varphi\right]dx=0,
\end{equation*}
where $n$ is the outward normal to $\partial \Omega$ and $C^{\infty}_c(\Omega)$ denoted the class of infinitely many times differentiable functions, with compact support in $\Omega$. Integrating by parts, we have
\begin{equation*}
\int_{\Omega}f(c)\varphi dx=0,
\end{equation*}
for any $\varphi\in C^{\infty}_c(\Omega)$, which leads to $f(c)=0$. Now, $f(0)=0$ and $f$ is strictly increasing. Hence $f(c)=0$ if and only if $c=0$, whence $y\equiv 0$ and $u\equiv 0$.
\end{proof}

%
In the next Lemma, $x^i$ and $\mathbf{e}^i$ denote respectively the $i_{\mbox{\tiny{th}}}$ component of the vector $x\in \mathbb{R}^n$ and the $i_{\mbox{\tiny{th}}}$ element of the canonical base of $\mathbb{R}^n$.
\begin{lemma}\label{lemma_posmeasndir}
Let $\Omega$ be an open set. Let $E\subset \Omega$ be a Lebesgue measurable set, with positive Lebesgue measure. Then for a.e. $\hat{x}\in \Omega$ and for any $i=1,\dots,n$ there exists a sequence $\left\{x_m^i\right\}_{m\in \mathbb{N}}\subset \mathbb{R}$ (with $x_m^i \mathbf{e}^i+\hat{x}\neq \hat{x}$) such that
\begin{equation}
x_m^i \mathbf{e}^i+\hat{x}\underset{m\to +\infty}{\longrightarrow}\hat{x}.
\end{equation}
\end{lemma}
\begin{proof}[Proof of \Cref{lemma_posmeasndir}]
%
Let us introduce the set of component-isolated points of $E$
\begin{equation}
E_{\mbox{\tiny{ci}}}\coloneqq \bigcup_{\substack{i=1,\dots,n \\ r>0} }\left\{x\in E \ | \ B_{\mbox{\tiny{i, r}}}\cap E=\left\{x\right\} \right\}.
\end{equation}
where
\begin{equation}
B_{\mbox{\tiny{i, r}}}\coloneqq\left\{\left(x^1,\dots,x^{i-1},y,x^{i+1},\dots ,x^{n}\right) \ | \ y\in [x^{i}-r,x^{i}+r]\right\}.
\end{equation}
\textit{Step 1} \  \textbf{Reduction to single component and radius}\\
By the the above definitions and the density of rational numbers in the reals, we have
\begin{equation}
E_{\mbox{\tiny{ci}}}= \bigcup_{\substack{i=1,\dots,n \\ r>0 \ \mbox{\tiny{and}} \ r\in \mathbb{Q} }}\left\{x\in E \ | \ B_{\mbox{\tiny{i, r}}}\cap E=\left\{x\right\} \right\}.
\end{equation}
By the countable additivity of the Lebesgue measure, we reduce then to prove that, for any $i=1,\dots,n$ and for any $r>0$, the set
\begin{equation}
E_{\mbox{\tiny{ci,i,r}}}\left\{x\in E \ | \ B_{\mbox{\tiny{i, r}}}\cap E=\left\{x\right\} \right\}
\end{equation}
is Lebesgue-measurable and has Lebesgue measure zero.\\
\textit{Step 2} \  \textbf{Conclusion}\\
The measurability of $E_{\mbox{\tiny{ci,i,r}}}$ follows from the continuity of the distance function. Let us compute its measure. For any $\left(x^1,\dots,x^{i-1},x^{i+1},\dots ,x^{n}\right)\in \mathbb{R}^{n-1}$, set
\begin{equation}
E_{\left(x^1,\dots,x^{i-1},x^{i+1},\dots ,x^{n}\right)}\coloneqq\left\{y\in \mathbb{R} \ | \ \left(x^1,\dots,x^{i-1},y,x^{i+1},\dots ,x^{n}\right)\in E_{\mbox{\tiny{ci,i,r}}} \right\},
\end{equation}
where we have dropped the subscript $_{\mbox{\tiny{ci,i,r}}}$ to avoid weighting the notation.

Now, by definition of $E_{\mbox{\tiny{ci,i,r}}}$, for any $x\in E_{\mbox{\tiny{ci,i,r}}}$, the set $E_{\left(x^1,\dots,x^{i-1},x^{i+1},\dots ,x^{n}\right)}$ is at most countable, whence of null Lebesgue measure. Then, by Fubini's Theorem, we have
\begin{equation}
\mu_{Leb}\left(E_{\mbox{\tiny{ci,i,r}}}\right) = \int_{\mathbb{R}^{n-1}}\mu_{Leb}\left(E_{\left(x^1,\dots,x^{i-1},x^{i+1},\dots ,x^{n}\right)}\right)d(x_1,\dots,x_{i-1},x_{i+1},\dots,x_n)=0,
\end{equation}
as required.
\end{proof}

\begin{lemma}\label{lemma_funtion_grad}
Let $\Omega$ be an open set. Let $y_1$ and $y_2$ be two functions of class $C^2\left(\Omega\right)$. Then,
\begin{equation}
\mu_{Leb}\left(\left\{x\in E \ | \ y_1(x)=y_2(x) \right\}\right)=\mu_{Leb}\left(\left\{x\in E \ | \ y_1(x)=y_2(x) \ \mbox{and} \ \nabla y_1(x)=\nabla y_2(x) \right\}\right).
\end{equation}
\end{lemma}
\begin{proof}
If $\mu_{Leb}\left(\left\{x\in E \ | \ y_1(x)=y_2(x) \right\}\right)$, the thesis follows.
Otherwise, let us apply \Cref{lemma_posmeasndir}, getting for a.e. $\hat{x}\in \Omega$ and for any $i=1,\dots,n$ a sequence $\left\{x_m^i\right\}_{m\in \mathbb{N}}\subset \mathbb{R}$ (with $x_m^i \mathbf{e}^i+\hat{x}\neq \hat{x}$) such that
\begin{equation*}
x_m^i \mathbf{e}^i+\hat{x}\underset{m\to +\infty}{\longrightarrow}\hat{x}.
\end{equation*}
Then,
\begin{eqnarray*}
	\frac{\partial y_2}{\partial x_i}\left(\hat{x}\right)&=&\lim_{m\to +\infty}\frac{y_2\left(x_m^i \mathbf{e}^i+\hat{x}\right)-y_2\left(\hat{x}\right)}{x_m^i \mathbf{e}^i+\hat{x}-\hat{x}}\nonumber\\
	&=&\lim_{m\to +\infty}\frac{y_1\left(x_m^i \mathbf{e}^i+\hat{x}\right)-y_1\left(\hat{x}\right)}{x_m^i \mathbf{e}^i+\hat{x}-\hat{x}}\nonumber\\
	&=&\frac{\partial y_1}{\partial x_i}\left(\hat{x}\right),\nonumber\\
\end{eqnarray*}
whence $\nabla y_2\left(\hat{x}\right)= \nabla y_1\left(\hat{x}\right)$, as required.
\end{proof}

\begin{lemma}\label{lemma_noncostcontrol}
Let $u\in L^{\infty}\left(\partial B(0,R)\right)$ be nonconstant. Then, there exists an orthogonal matrix $M$, such that
\begin{equation}
u\circ M \neq u.
\end{equation}
\end{lemma}
\begin{proof}[Proof of \Cref{lemma_noncostcontrol}]
In the present proof, we denote by $\tilde{u}$ a representative of the equivalence class $u\in L^{\infty}\left(\partial B(0,R)\right)$.
By \cite[Theorem 7.7]{rudin2006real}, a.e. $x\in \partial B(0,R)$ is a Lebesgue point for $\tilde{u}$, whence there exists $x_1\neq x_2$ Lebesgue points such that $\tilde{u}(x_1)\neq \tilde{u}(x_2)$. Let $M$ be an orthogonal matrix such that $Mx_1=x_2$. Then, since $x_1$ and $x_2$ are Lebesgue points, there exists $r>0$ such that
\begin{equation}\label{lemma_noncostcontrol_eq3}
\int_{\partial B(0,R)\cap B(x_1,r)}\tilde{u}_M(x)dx=\int_{\partial B(0,R)\cap B(x_2,r)}\tilde{u}(y)dy\neq \int_{\partial B(0,R)\cap B(x_1,r)}\tilde{u}(x)dx,
\end{equation}
where we have used the change of variable $y\coloneqq Mx$ and $\tilde{u}_M(x)\coloneqq \tilde{u}(Mx)$. \cref{lemma_noncostcontrol_eq3} shows that $u\circ M\neq u$, as required.
\end{proof}

We state and prove a well-known result: the rotational invariance of the Laplacian.
\begin{lemma}\label{lemma_rotinv_laplacian}
Let $\varphi\in C^2\left(\Omega\right)$ and let $M$ be an $n\times n$ orthogonal matrix. Then, for any $x\in \Omega$
\begin{equation}
\Delta\left(\varphi\circ M\right)=\Delta \left(\varphi\right)\circ M \hspace{0.6 cm} \mbox{in} \hspace{0.10 cm}\Omega.
\end{equation}
\end{lemma}
\begin{proof}[Proof of \Cref{lemma_rotinv_laplacian}]
By the chain rule and the orthogonality of $M$, we have
\begin{equation*}
\mbox{Hess}\left(\varphi\circ M\right)=M^{-1}\left[\mbox{Hess}\left(\varphi\right)\circ M\right]M,
\end{equation*}
whence, by the similarity invariance of the trace, for any $x\in \Omega$
\begin{equation*}
\Delta\left(\varphi\circ M\right)=\mbox{Trace}\left(\mbox{Hess}\left(\varphi\circ M\right)\right)=\mbox{Trace}\left(M^{-1}\left[\mbox{Hess}\left(\varphi\right)\circ M\right]M\right)=\Delta \left(\varphi\right)\circ M,
\end{equation*}
as required.
\end{proof}

\begin{lemma}\label{lemma_rot}
Consider a rotational invariant domain $\Omega$. Let $u\in L^{\infty}(\partial \Omega)$ be a control and let $y$ be the solution to \eqref{semilinear_boundary_elliptic_1_gdomain_wp}, with control $u$. Let $M$ be an orthogonal matrix. Set $u_M(x)\coloneqq u(M(x))$ and $y_M(x)\coloneqq y(M(x))$. Then, $y_M$ is a solution to
\begin{equation}\label{semilinear_boundary_elliptic_rot_gdomain_wp}
\begin{dcases}
-\Delta y_M+f\left(y_M\right)=0\hspace{2.8 cm} & \mbox{in} \hspace{0.10 cm}\Omega\\
y_M=u_M  & \mbox{on}\hspace{0.10 cm} \partial \Omega
\end{dcases}
\end{equation}
in the sense of \Cref{def_sol_semilinear_boundary_elliptic}. If in addition $u_M=u$ for any orthogonal matrix $M$, then $y_M=y$, namely $y$ is a radial solution.
\end{lemma}
\begin{proof}[Proof of \Cref{lemma_rot}]
As per \Cref{def_sol_semilinear_boundary_elliptic}, let us check that for any test function $\varphi\in \mathscr{C}$, we have
\begin{equation}\label{lemma_rot_eq3}
\int_{\Omega}\left[-y_M(x)\Delta \varphi(x)+f\left(y_M(x)\right)\varphi(x)\right]dx+\int_{\partial \Omega}u_M(x)\frac{\partial \varphi(x)}{\partial n} d\sigma(x)=0.
\end{equation}
Set $\tilde{x}\coloneqq Mx$. Since the matrix $M$ is orthogonal, $\left|\det(M)\right|=1$, whence by Change of Variables Theorem, definition of $y_M$ and \Cref{lemma_rotinv_laplacian}
\begin{eqnarray}\label{lemma_rot_eq6}
&\;&\int_{\Omega}\left[-y_M(x)\Delta \varphi(x)+f\left(y_M(x)\right)\varphi(x)\right]dx\nonumber\\
&=&\int_{\Omega}\left[-y\left(\tilde{x}\right)\Delta_{x} \varphi\left(M^{-1}\tilde{x}\right)+f\left(y\left(\tilde{x}\right)\right)\varphi\left(M^{-1}\tilde{x}\right)\right]d\tilde{x}\nonumber\\
&=&\int_{\Omega}\left[-y\left(\tilde{x}\right)\Delta_{\tilde{x}} \varphi\left(M^{-1}\tilde{x}\right)+f\left(y\left(\tilde{x}\right)\right)\varphi\left(M^{-1}\tilde{x}\right)\right]d\tilde{x}\nonumber\\
&=&\int_{\partial \Omega}u\left(\tilde{x}\right)\nabla_{\tilde{x}} \varphi(M^{-1}\tilde{x})\cdot n\left(\tilde{x}\right) d\sigma\left(\tilde{x}\right)\nonumber,\\
\end{eqnarray}
where in the last inequality we have used that $y$ is a solution to \eqref{semilinear_boundary_elliptic_1_gdomain_wp}, with control $u$. Now, we change back variable $x\coloneqq M^{-1}\tilde{x}$ in \eqref{lemma_rot_eq6}, getting
\begin{equation}
\int_{\partial \Omega}u\left(\tilde{x}\right)\nabla_{\tilde{x}} \varphi(M^{-1}\tilde{x})\cdot n\left(\tilde{x}\right) d\sigma\left(\tilde{x}\right)=\int_{\partial \Omega}u\left(Mx\right)\nabla_{x} \varphi(x)M^{-1}\cdot Mn(x) d\sigma\left(\tilde{x}\right),
\end{equation}
whence \cref{lemma_rot_eq3} follows. Therefore, if the control is radial, for any orthogonal matrix $M$, $y_M$ is the solution to the same boundary value problem. The uniqueness for \cref{semilinear_boundary_elliptic_1_gdomain_wp} yields $y_M=y$.
\end{proof}

We now prove the existence of a global minimizer for the functional $J$, defined in \cref{semilinear_boundary_elliptic_1}-\cref{functional_nouniqboundary}. This will be given by the coercivity in $L^2$ of $J$, enhanced by employing the regularity of the solutions to the optimality system. As we did in the former section, we are going to accomplish this task in a general space domain $\Omega$. Consider the optimal control problem
\begin{equation}\label{functional_nouniqboundary_gdomain}
\min_{u\in L^{\infty}(\partial \Omega)}J(u)=\frac12\int_{\partial \Omega} |u|^2 d\sigma(x) +\frac{\beta}{2}\int_{\Omega} |y-z|^2 dx,
\end{equation}
where:
\begin{equation}\label{semilinear_boundary_elliptic_1_gdomain}
\begin{dcases}
-\Delta y+f(y)=0\hspace{2.8 cm} & \mbox{in} \hspace{0.10 cm}\Omega\\
y=u  & \mbox{on}\hspace{0.10 cm} \partial \Omega.
\end{dcases}
\end{equation}
$\Omega$ is a bounded open subset of $\mathbb{R}^n$, with $n=1,2,3$ and $\partial \Omega\in C^{\infty}$. The nonlinearity $f\in C^1\left(\mathbb{R}\right)\cap C^2\left(\mathbb{R}\setminus \left\{0\right\}\right)$ is strictly increasing and $f(0)=0$. The target $z\in L^{\infty}(\Omega)$ and $\beta> 0$ is a penalization parameter.

\begin{proposition}\label{proposition_eolinf}
Let $z\in L^{\infty}(\Omega)$ be target for the state and let $J$ be the corresponding functional, defined in \eqref{semilinear_boundary_elliptic_1_gdomain}-\eqref{functional_nouniqboundary_gdomain}. There exists $\overline{u}\in L^{\infty}(\partial \Omega)$ a global minimizer for $J$.
\end{proposition}
\begin{proof}[Proof of \Cref{proposition_eolinf}]
\textit{Step 1} \ \textbf{Existence of the minimizer for a constrained problem}\\
Let $a$, $b \ \in \mathbb{R}$, with $a<0<b$ and let the convex set
\begin{equation*}
\mathbb{K}\coloneqq \left\{u \in L^{\infty}\left(\partial \Omega\right) \ | \ a\leq u\leq b, \ \mbox{a.e.} \ \partial \Omega \right\}.
\end{equation*}
Under the same assumptions of \cref{semilinear_boundary_elliptic_1_gdomain}-\cref{functional_nouniqboundary_gdomain}, we consider the constrained optimal control problem:
\begin{equation}\label{functional_nouniqboundary_gdomain_constrained}
\min_{u\in \mathbb{K}}J(u)=\frac12\int_{\partial \Omega} |u|^2 d\sigma(x) +\frac{\beta}{2}\int_{\Omega} |y-z|^2 dx,
\end{equation}
where:
\begin{equation}\label{semilinear_boundary_elliptic_1_gdomain_constrained}
\begin{dcases}
-\Delta y+f(y)=0\hspace{2.8 cm} & \mbox{in} \hspace{0.10 cm}\Omega\\
y=u  & \mbox{on}\hspace{0.10 cm} \partial \Omega.
\end{dcases}
\end{equation}
By using the techniques in \cite{ESC}, we have the existence of an optimal control $\overline{u}_{\left(a,b\right)}\in \mathbb{K}$ and any optimal control is given by $\overline{u}_{\left(a,b\right)}=\mathbb{P}_{\left[a,b\right]}\left(\frac{\partial \overline{q}_{(a,b)}}{\partial n}\right)$, with
\begin{equation}\label{OS_constrained}
\begin{dcases}
-\Delta \overline{y}_{(a,b)}+f(\overline{y}_{(a,b)})=0\hspace{2.8 cm} & \mbox{in} \hspace{0.10 cm}\Omega\\
\overline{y}_{(a,b)}=\mathbb{P}_{\left[a,b\right]}\left(\frac{\partial \overline{q}_{(a,b)}}{\partial n}\right)  & \mbox{on}\hspace{0.10 cm} \partial \Omega\\
-\Delta \overline{q}_{(a,b)}+f^{\prime}(\overline{y}_{(a,b)})\overline{q}_{(a,b)}=\beta\left(\overline{y}_{(a,b)}-z\right)\hspace{2.8 cm} & \mbox{in} \hspace{0.10 cm}\Omega\\
\overline{q}_{(a,b)}=0  & \mbox{on}\hspace{0.10 cm} \partial \Omega,
\end{dcases}
\end{equation}
where $\mathbb{P}_{\left[a,b\right]}$ is the projector
\begin{equation}\mathbb{P}_{\left[a,b\right]}(\xi)\coloneqq
\begin{dcases}
a \hspace{0.3 cm} &\mbox{if} \ \xi\leq a\\
\xi  &\mbox{if} \ a<\xi< b\\
b  &\mbox{if} \ \xi\geq b.\\
\end{dcases}
\end{equation}
\textit{Step 2} \ \textbf{$L^{\infty}$ bounds for optimal controls uniform on $\left(a,b\right)\in\mathbb{R}^2$, with $a<0<b$}\\
Since $a<0<b$, the null control $0\in \mathbb{K}$. Then, for any optimal control $\overline{u}_{\left(a,b\right)}$ for \cref{semilinear_boundary_elliptic_1_gdomain_constrained}-\cref{functional_nouniqboundary_gdomain_constrained}, we have
\begin{equation*}
\frac12\int_{\partial \Omega}\left|\overline{u}_{\left(a,b\right)}\right|^2d \sigma(x)\leq J\left(\overline{u}_{\left(a,b\right)}\right)\leq J\left(0\right)\leq K,
\end{equation*}
whence
\begin{equation}\label{L^2_bound}
\left\|\overline{u}_{\left(a,b\right)}\right\|_{L^2(\partial \Omega)}\leq K,
\end{equation}
where $K=K(\Omega,f,\beta,z)$ is independent of $\left(a,b\right)$.

We now bootstrap in the optimality system \cref{OS_constrained}, to get the desired $L^{\infty}$ bound, given the above $L^2$ bound.

First of all, by a comparison argument, we have
\begin{equation}\label{comparison_state}
\left|\overline{y}_{(a,b)}\right|\leq \hat{y}_{\left(a,b\right)}, \hspace{0.3 cm}\mbox{a.e.} \ \Omega,
\end{equation}
with
\begin{equation}\label{upper_semilinear_boundary_elliptic_1_gdomain_constrained}
\begin{dcases}
-\Delta \hat{y}_{\left(a,b\right)}=0\hspace{2.8 cm} & \mbox{in} \hspace{0.10 cm}\Omega\\
\hat{y}_{\left(a,b\right)}=\left|\overline{u}_{\left(a,b\right)}\right|  & \mbox{on}\hspace{0.10 cm} \partial \Omega.
\end{dcases}
\end{equation}
Comparison gives also
\begin{equation}\label{comparison_adjoint}
\left|\overline{q}_{(a,b)}\right|\leq \hat{q}_{\left(a,b\right)}\hspace{0.3 cm}\mbox{and}\hspace{0.3 cm}\left|\frac{\partial \overline{q}_{(a,b)}}{\partial n}\right|\leq \left|\frac{\partial \hat{q}_{\left(a,b\right)}}{\partial n}\right|, \hspace{0.3 cm}\mbox{a.e.} \ \Omega
\end{equation}
with
\begin{equation}\label{upper_adjoint_boundary_elliptic_1_constrained}
\begin{dcases}
-\Delta \hat{q}_{\left(a,b\right)}=\beta\left|\overline{y}_{(a,b)}-z\right|\hspace{2.8 cm} & \mbox{in} \hspace{0.10 cm}\Omega\\
\hat{q}_{\left(a,b\right)}=0  & \mbox{on}\hspace{0.10 cm} \partial \Omega.
\end{dcases}
\end{equation}

Now, by \cite[Th\'eor\`eme 7.4, page 202]{LM1}, the solution $\hat{y}_{\left(a,b\right)}\in H^{\frac12}\left(\Omega\right)\hookrightarrow L^3(\Omega)$ and
\begin{equation*}
\left\|\hat{y}_{\left(a,b\right)}\right\|_{L^3(\Omega)}\leq K\left\|\hat{y}_{\left(a,b\right)}\right\|_{H^{\frac12}(\Omega)}\leq K\left\|\overline{u}_{\left(a,b\right)}\right\|_{L^{2}(\partial \Omega)}\leq K.
\end{equation*}
where the first inequality is given by the Sobolev embedding $H^{\frac12}\left(\Omega\right)\hookrightarrow L^3(\Omega)$ valid for space dimension $n=1,2,3$ (see e.g. \cite[Theorem 6.7]{HFV}) and the last inequality is justified by \cref{L^2_bound}. By \cref{comparison_state},
\begin{equation*}
\left\|\overline{y}_{(a,b)}\right\|_{L^3(\Omega)}\leq \left\|\hat{y}_{\left(a,b\right)}\right\|_{L^3(\Omega)}\leq K.
\end{equation*}

We now concentrate on the adjoint equation. By \cite[Theorem 2.4.2.5 page 124]{GNE} applied to \cref{upper_adjoint_boundary_elliptic_1_constrained}, we have $\hat{q}_{\left(a,b\right)}\in W^{2,3}(\Omega)$, with estimate
\begin{equation*}
\left\|\hat{q}_{\left(a,b\right)}\right\|_{W^{2,3}(\Omega)}\leq K\left\|\overline{y}_{(a,b)}-z\right\|_{L^3(\Omega)}\leq K\left[\left\|\overline{y}_{(a,b)}\right\|_{L^3(\Omega)}+\|z\|_{L^{\infty}(\Omega)}\right]\leq K.
\end{equation*}
By the trace Theorem (\cite[Theorem 1.5.1.3 page 38]{GNE}) applied to $\nabla \hat{q}_{\left(a,b\right)}$,
\begin{equation*}
\left\|\frac{\partial \hat{q}_{\left(a,b\right)}}{\partial n}\right\|_{L^{4}(\partial \Omega)}\leq K \left\|\hat{q}_{\left(a,b\right)}\right\|_{W^{2,3}(\Omega)}\leq K.
\end{equation*}
By \cref{comparison_adjoint}, we have then
\begin{equation}\label{est_L4_contr}
\left\|\frac{\partial \overline{q}_{(a,b)}}{\partial n}\right\|_{L^{4}(\partial \Omega)}\leq \left\|\frac{\partial \hat{q}_{\left(a,b\right)}}{\partial n}\right\|_{L^{4}(\partial \Omega)}\leq K,
\end{equation}
whence
\begin{equation*}
\left\|\overline{u}_{\left(a,b\right)}\right\|_{L^{4}(\partial \Omega)}=\left\|\mathbb{P}_{\left[a,b\right]}\left(\frac{\partial \overline{q}_{(a,b)}}{\partial n}\right)\right\|_{L^{4}(\partial \Omega)}\leq \left\|\frac{\partial \overline{q}_{(a,b)}}{\partial n}\right\|_{L^{4}(\partial \Omega)}\leq K.
\end{equation*}
By using the definition of solution by transposition for \cref{upper_semilinear_boundary_elliptic_1_gdomain_constrained} and the above estimate, we get
\begin{equation*}
\left\|\hat{y}_{\left(a,b\right)}\right\|_{L^{4}(\Omega)}\leq K \left\|\overline{u}_{\left(a,b\right)}\right\|_{L^{4}(\partial \Omega)}\leq K,
\end{equation*}
whence, by \cref{comparison_state}
\begin{equation*}
\left\|\overline{y}_{\left(a,b\right)}\right\|_{L^{4}(\Omega)}\leq \left\|\hat{y}_{\left(a,b\right)}\right\|_{L^{4}(\Omega)}\leq K.
\end{equation*}

In conclusion, we employ the elliptic regularity (\cite[Theorem 2.4.2.5 page 124]{GNE}) in \cref{upper_adjoint_boundary_elliptic_1_constrained}, to get
\begin{equation*}
\left\|\hat{q}_{\left(a,b\right)}\right\|_{W^{2,4}(\Omega)}\leq K\left\|\overline{y}_{\left(a,b\right)}-z\right\|_{L^4(\Omega)}\leq K,
\end{equation*}
whence, by Sobolev embeddings in space dimension $n=1,2,3$,
\begin{equation*}
\left\|\hat{q}_{\left(a,b\right)}\right\|_{C^{1}\left(\overline{\Omega}\right)}\leq \left\|\hat{q}_{\left(a,b\right)}\right\|_{W^{2,4}(\Omega)}\leq K\left\|y-z\right\|_{L^4(\Omega)}\leq K.
\end{equation*}
Now, \cref{comparison_adjoint} yields
\begin{equation}\label{proposition_eolinf_eq636}
\left\|\frac{\partial\overline{q}_{\left(a,b\right)}}{\partial n}\right\|_{C^{0}\left(\partial \Omega\right)}\leq \left\|\frac{\partial \hat{q}_{\left(a,b\right)}}{\partial n}\right\|_{C^{0}\left(\partial \Omega\right)}\leq \left\|\hat{q}_{\left(a,b\right)}\right\|_{C^{1}\left(\overline{\Omega}\right)}\leq K,
\end{equation}
which in turn implies
\begin{equation*}
\left\|\overline{u}_{\left(a,b\right)}\right\|_{L^{\infty}(\partial \Omega)}=\left\|\mathbb{P}_{\left[a,b\right]}\left(\frac{\partial \overline{q}_{(a,b)}}{\partial n}\right)\right\|_{L^{\infty}(\partial \Omega)}\leq \left\|\frac{\partial \overline{q}_{(a,b)}}{\partial n}\right\|_{L^{\infty}(\partial \Omega)}\leq K,
\end{equation*}
where the last inequality follows from \cref{proposition_eolinf_eq636}. We have then, the estimate
\begin{equation}\label{proposition_eolinf_eq639}
\left\|\overline{u}_{\left(a,b\right)}\right\|_{L^{\infty}(\partial \Omega)}\leq K,\hspace{0.3 cm} \forall \ a, \ b \in \mathbb{R}, \ \mbox{with} \ a<0<b,
\end{equation}
the constant $K=K(\Omega,f,\beta,z)$ being independent of $\left(a,b\right)$. This finishes this step.\\
\textit{Step 3} \  \textbf{Conclusion}\\
Let $K$ be the upper bound appearing in \cref{proposition_eolinf_eq639}. We want to show that, for any control $u\in L^{\infty}(\partial \Omega)$, with $\left\|u\right\|_{L^{\infty}(\partial \Omega)}>K$, the value of the functional
\begin{equation*}
J(u)>\inf_{B^{L^{\infty}}(0,K)}J,
\end{equation*}
Indeed, for any control $u\in L^{\infty}(\partial \Omega)$, with $\left\|u\right\|_{L^{\infty}(\partial \Omega)}>K$, set $b\coloneqq \left\|u\right\|_{L^{\infty}(\partial \Omega)}+1$, $a\coloneqq -b$ and set accordingly the control set
\begin{equation*}
\mathbb{K}\coloneqq \left\{u \in L^{\infty}\left(\partial \Omega\right) \ | \ a\leq u\leq b, \ \mbox{a.e.} \ \partial \Omega \right\}.
\end{equation*}
By definition of $a$ and $b$, the control $u\in \mathbb{K}$ and, by \cref{proposition_eolinf_eq639}
\begin{equation}\label{proposition_eolinf_eq640}
J(u)>\inf_{B^{L^{\infty}}(0,K)}J,
\end{equation}
as desired. Now, by step 1, there exists $\overline{u}\in \overline{B^{L^{\infty}}(0,K)}$ minimizing $J$ in $\overline{B^{L^{\infty}}(0,K)}$. By \cref{proposition_eolinf_eq640}, such control $\overline{u}$ is in fact a global minimizer for $J$ in $L^{\infty}(\partial \Omega)$, thus concluding the proof.
\end{proof}

\begin{proof}[Proof of \Cref{lemma3_boundary}.]
\textit{Step 1} \ \textbf{Proof of 1.}\\
Arbitrarily fix $z\in L^{\infty}(B(0,R))$. The existence of a minimizer $u_{z}$ is a consequence of the direct methods in the Calculus of Variations. Moreover, by \cref{def_functional_control_target_PDE_boundary}, definition of minimizer and $G(0)=0$:
\begin{multline*}
\frac12R^{n-1}n\alpha(n)|u_{z}|^2\leq I(u_{z},z)+\frac{\beta}{2}\int_{B(0,R)}|z|^2dx\\
\leq I(0,z)+\frac{\beta}{2}  \int_{B(0,R)}|z|^2dx=\frac{\beta}{2} \int_{B(0,R)}|z|^2dx,
\end{multline*}
which yields $\frac12|u_{z}|^2\leq \frac{\beta}{2R^{n-1}n\alpha(n)} \int_{B(0,R)}|z|^2dx$, as required.\\
\textit{Step 2} \ \textbf{Proof of 2.}\\
Arbitrarily fix $M\in\mathbb{R}^+$. For any pair of targets $(z_1,z_2)\in L^{\infty}(B(0,R))^2$ such that:
\begin{equation*}
\|z_1\|_{L^{2}}\leq M \hspace{1 cm}\mbox{and}\hspace{1 cm}\|z_2\|_{L^2}\leq M.
\end{equation*}
For each control $u\in C$ such that $|u|\leq \sqrt{\frac{\beta}{R^{n-1}n\alpha(n)}} M$, we have:
\begin{equation*}
I(u,z_2)-I(u_{z_1},z_1)=I(u,z_2)-I(u,z_1)+I(u,z_1)-I(u_{z_1},z_1)
\end{equation*}
\begin{equation*}
\geq -|I(u,z_2)-I(u,z_1)|+0=-\beta\left|\int_{B(0,R)} G(u)(z_1-z_2)dx\right|
\end{equation*}
\begin{equation*}
\geq -K\|z_2-z_1\|_{L^{\infty}},
\end{equation*}
where the last inequality is justified by $|u|\leq \sqrt{\frac{\beta}{R^{n-1}n\alpha(n)}} M$ and the continuity of the control-to-state map $G$.

Then, one has that for any $\varepsilon >0$, there exists $\delta_{\varepsilon}>0$ such that:
\begin{equation*}
I(u,z_2)-I(u_{z_1},z_1)>-\varepsilon,
\end{equation*}
whenever $\|z_2-z_1\|_{L^{\infty}}<\delta_{\varepsilon}$.

Now, by the first step, any minimizer $u_{z_2}$ for $I(\cdot,z_2)$ verifies\\
$|u_{z_2}|\leq\sqrt{\frac{\beta}{R^{n-1}n\alpha(n)}} \|z_2\|_{L^2}\leq\sqrt{\frac{\beta}{R^{n-1}n\alpha(n)}} M$. Then, we have proved that:
\begin{equation*}
\inf_{C}[I(\cdot,z_2)]-\inf_C [I(\cdot,z_1)]=I(u_{z_2},z_2)-I(u_{z_1},z_1)>-\varepsilon.
\end{equation*}
Exchanging the role of $z_1$ and $z_2$, one can get:
\begin{equation*}
\inf_{C}[I(\cdot,z_1)]-\inf_C[I(\cdot, z_2)]>-\varepsilon.
\end{equation*}
This yields the continuity of $h$.	
\end{proof}

\begin{proof}[Proof of \Cref{lemma4_boundary}.]
If $h_1(z^0)=h_2(z^0)$, we take $\tilde{z}\coloneqq z^0$, thus concluding. Let us now suppose $h_1(z^0)\neq h_2(z^0)$.

We start by considering the case $h_1(z^0)<h_2(z^0)$.\\
\textit{Step 1} \ \textbf{Proof of the existence of $\mu_0\geq 0$ such that:
	\begin{itemize}
		\item $\forall \mu \in [0,\mu_0]$, $h_2(z^0+\mu)<0$;
		\item $
		h_1\left(z^0+\mu_0 
		\right)=0.
		$
\end{itemize}}
First of all, we observe that for any $\mu\geq 0$, $h_2(z^0+\mu
)<0$. Indeed, since $h_2(z^0)<0$, there exists $u_2>0$ such that $I(u_2,z^0)<0$. Then,
\begin{multline*}
h_2(z^0+\mu)\leq I(u_2,z^0+\mu)\\
=\frac{R^{n-1}n\alpha(n)}{2}|u_2|^2+\frac{\beta}{2}\int_{B(0,R)} |G(u_2)|^2dx- \beta\int_{B(0,R)}(z^0+\mu)G(u_2)dx\\
=I(u_2,z^0)-\mu \beta \int_{B(0,R)}G(u_2)dx\leq I(u_2,z^0)<0,
\end{multline*}
where we have used that $G(u_2)\geq 0$ a.e. in $B(0,R)$.

We prove now that $h_1\left(z^0+\mu_0 
\right)=0$, for $\mu_0=\|z^0\|_{L^{\infty}}$. Indeed, for any $v\leq 0$:
\begin{equation*}
I(v,z^0+\mu_0
)=\frac{R^{n-1}n\alpha(n)}{2}|v|^2+\frac{\beta}{2}\int_{B(0,R)} |G(v)|^2dx-\beta\int_{B(0,R)}(z^0+\mu_0) G(v)dx\geq 0,
\end{equation*}	
since $z^0+\mu_0\geq 0$ and $G(v)\leq 0$ a.e. in $B(0,R)$. This finishes the first step.\\
\textit{Step 2} \ \textbf{Conclusion}\\
Set:
\begin{equation*}
g:[0,\mu_0]\longrightarrow \mathbb{R}
\end{equation*}
\begin{equation*}
\mu\longmapsto h_2(z^0+\mu )-h_1(z^0+\mu ).
\end{equation*}
Since $h_1(z^0)<h_2(z^0)$, $g(0)>0$ and by Step 1 $g(\mu_0)<0$. Then, by continuity, there exists $\mu_1\in (0,\mu_0)$ such that $g(\mu_1)=0$. Hence,
\begin{equation*}
\tilde{z} \coloneqq z^0+\mu_1
\end{equation*}
is the desired target. Indeed, by definition of $g$ and $\mu_1$, $h_1(\tilde{z})=h_2(\tilde{z})$. Furthermore, since $\mu_1\in (0,\mu_0)$, by Step 1, $h_2(\tilde{z})<0$. This concludes the proof for the case $h_1(z^0)<h_2(z^0)$. The proof for the remaining case $h_1(z^0)>h_2(z^0)$ is similar.
\end{proof}

\section{Preliminaries for internal control}
\label{sec:appendix.Preliminaries for internal control}

We consider now study the state equation \cref{semilinear_internal_elliptic_1} on a general domain. Let $\Omega$ be an bounded open subset of $\mathbb{R}^n$, with $\partial \Omega \in C^2$ and $n=1,2,3$. The nonlinearity $f\in C^1\left(\mathbb{R}\right)\cap C^2\left(\mathbb{R}\setminus \left\{0\right\}\right)$ is strictly increasing and $f(0)=0$. The control acts in $\omega$, nonempty open subset of $\Omega$.

We introduce the concept of solution, following \cite[Theorem 4.7, page 29]{boccardo2013elliptic}.

\begin{definition}\label{def_sol_semilinear_internal_elliptic}
Let $u\in L^{2}(\omega)$. Then, $y\in H^1_0(\Omega)$ is said to be a solution to
\begin{equation}\label{semilinear_internal_elliptic_1_generaldomain}
\begin{dcases}
-\Delta y+f(y)=u\chi_{\omega}\hspace{2.8 cm} & \mbox{in} \hspace{0.10 cm}\Omega\\
y=0  & \mbox{on}\hspace{0.10 cm} \partial \Omega.
\end{dcases}
\end{equation}
if $f(y)\in L^1(\Omega)$ and for any test function $\varphi\in H^1_0(\Omega)\cap L^{\infty}(\Omega)$, we have
\begin{equation*}
\int_{\Omega}\left[\nabla y\cdot \nabla \varphi+f(y)\varphi\right]dx=\int_{\omega}u\varphi dx.
\end{equation*}
\end{definition}

The well-posedness of \cref{semilinear_internal_elliptic_1_generaldomain} follows from \cite[Theorem 4.7, page 29]{boccardo2013elliptic}.

\begin{lemma}\label{lemma_nonconst_internal}
Let $u\in L^{\infty}(\omega)$ be a control. Let $y$ be the solution to \eqref{semilinear_internal_elliptic_1_generaldomain}, with control $u$. Assume the nonlinearity $f$ is strictly increasing and $y$ is constant in $\Omega\setminus \omega$. Then, $y\equiv 0$ and $u\equiv 0$.
\end{lemma}
\begin{proof}[Proof of \Cref{lemma_nonconst_internal}]
Suppose there exists $c\in\mathbb{R}$, such that $y(x)=c$, for any $x\in \Omega\setminus\omega$. Then, by \Cref{def_sol_semilinear_boundary_elliptic}, for any for any test function $\varphi\in C^{\infty}_c(\Omega\setminus \omega)$, we have
\begin{equation*}
\int_{\Omega}f(c)\varphi dx=\int_{\Omega}\left[\nabla y\cdot \nabla \varphi+f(y)\varphi\right]dx=\int_{\omega}u\varphi dx=0,
\end{equation*}
where $C^{\infty}_c(\Omega\setminus \omega)$ denoted the class of infinitely many times differentiable functions, with compact support in $\Omega\setminus \omega$. The arbitrariness of $\varphi\in C^{\infty}_c(\Omega\setminus \omega)$ leads to $f(c)=0$. Now, $f(0)=0$ and $f$ is strictly increasing. Hence $f(c)=0$ if and only if $c=0$, whence $y\equiv 0$ and $u\equiv 0$.
\end{proof}

\begin{lemma}\label{lemma_rot_int}
In the notation of \cref{semilinear_internal_elliptic_1_generaldomain}, consider rotational invariant domains $\Omega$ and $\omega$. Let $u\in L^{\infty}(\omega)$ be a control and let $y$ be the solution to \eqref{semilinear_boundary_elliptic_1_gdomain_wp}, with control $u$. Let $M$ be an orthogonal matrix. Set $u_M(x)\coloneqq u(M(x))$ and $y_M(x)\coloneqq y(M(x))$. Then, $y_M$ is a solution to
\begin{equation}\label{semilinear_internal_elliptic_rot_generaldomain}
\begin{dcases}
-\Delta y_M+f\left(y_M\right)=u_M\chi_{\omega}\hspace{2.8 cm} & \mbox{in} \hspace{0.10 cm}\Omega\\
y_M=0  & \mbox{on}\hspace{0.10 cm} \partial \Omega
\end{dcases}
\end{equation}
in the sense of \Cref{def_sol_semilinear_internal_elliptic}. If in addition $u_M=u$ for any orthogonal matrix $M$, then $y_M=y$, namely $y$ is a radial solution.
\end{lemma}
\begin{proof}[Proof of \Cref{lemma_rot_int}]
As per \Cref{def_sol_semilinear_internal_elliptic}, let us check that for any test function $\varphi\in H^1_0(\Omega)\cap L^{\infty}(\Omega)$, we have
\begin{equation}\label{lemma_rot_int_eq3}
\int_{\Omega}\left[\nabla y_M\cdot \nabla \varphi+f\left(y_M\right)\varphi\right]dx=\int_{\omega}u_M\varphi dx.
\end{equation}
Set $\tilde{x}\coloneqq Mx$. Since the matrix $M$ is orthogonal, $\left|\det(M)\right|=1$, whence by Change of Variables Theorem and definition of $y_M$
\begin{eqnarray}\label{lemma_rot_int_eq6}
&\;&\int_{\Omega}\left[\nabla y_M\cdot \nabla \varphi+f\left(y_M\right)\varphi\right]dx\nonumber\\
&=&\int_{\Omega}\left[\left(\nabla_{\tilde{x}} y(Mx)M\right)\cdot \nabla \varphi+f\left(y_M\right)\varphi\right]dx\nonumber\\
&=&\int_{\Omega}\left[\nabla_{\tilde{x}} y(Mx)\cdot \left(\nabla_x \varphi(x) M^{-1}\right)+f\left(y_M\right)\varphi\right]dx\nonumber\\
&=&\int_{\Omega}\left[\nabla_{\tilde{x}} y(\tilde{x})\cdot \nabla_{\tilde{x}} \varphi\left(M^{-1}\tilde{x}\right)+f\left(y_M\left(M^{-1}\tilde{x}\right)\right)\varphi\left(M^{-1}\tilde{x}\right)\right]d\tilde{x}\nonumber\\
&=&\int_{\Omega}\left[\nabla_{\tilde{x}} y\left(\tilde{x}\right)\cdot \nabla_{\tilde{x}} \varphi\left(M^{-1}\tilde{x}\right)+f\left(y\left(\tilde{x}\right)\right)\varphi\left(M^{-1}\tilde{x}\right)\right]d\tilde{x}\nonumber\\
&=&\int_{\omega}u\left(\tilde{x}\right)\varphi\left(M^{-1}\tilde{x}\right) d\tilde{x},\\
\end{eqnarray}
where in the last inequality we have used that $y$ is a solution to \eqref{semilinear_boundary_elliptic_1_gdomain_wp}, with control $u$. Now, we change back variable $x\coloneqq M^{-1}\tilde{x}$ in \eqref{lemma_rot_int_eq6}, getting
\begin{equation}
\int_{\omega}u\left(\tilde{x}\right)\varphi\left(M^{-1}\tilde{x}\right) d\tilde{x}=\int_{\omega}u\left(Mx\right)\varphi\left(x\right) dx=\int_{\omega}u_M\left(x\right)\varphi\left(x\right) dx,
\end{equation}
whence \cref{lemma_rot_int_eq3} follows. Therefore, if the control is radial, for any orthogonal matrix $M$, $y_M$ is the solution to the same boundary value problem. The uniqueness for \cref{semilinear_internal_elliptic_1_generaldomain} yields $y_M=y$.
\end{proof}

\bigskip
\footnotesize
\noindent\textit{Acknowledgments.}
This project has received funding from the European Research Council (ERC) under the European Union’s Horizon 2020 research and innovation programme (grant agreement No 694126-DYCON).

We acknowledge professor Enrique Zuazua for his helpful remarks on the manuscript. We thank professor Martin Gugat for his interesting questions. We gratefully acknowledge the referees for their interesting comments.

\end{document}